%% file: main.tex
\documentclass{article}

\usepackage[a4paper,top=2cm,bottom=2cm,left=3cm,right=3cm,marginparwidth=1.75cm]{geometry}
\usepackage[OT1]{fontenc}
\usepackage[english]{babel}
\usepackage[utf8]{inputenc}
\usepackage{amsmath,amssymb,amsfonts,mathrsfs,tikz-cd,bbold,dsfont,tocloft}
\usepackage{hyperref}
\usepackage[amsmath,thmmarks,amsthm]{ntheorem}
\usepackage{graphicx,xcolor}
\usepackage{soul}
\usepackage{pdfpages}
\usepackage[all]{xy}
\usepackage{verbatim}
\usepackage{stmaryrd}
\usepackage{sectsty}

\hypersetup{linkcolor=black,colorlinks=true,citecolor=black,filecolor=black}

\makeatletter
\def\@seccntformat#1{\csname the#1\endcsname.\;\;}
\makeatother

\sectionfont{\centering}

\cftsetindents{section}{0em}{2em}
\cftsetindents{subsection}{0em}{2em}

\title{\textbf{The Morel-Voevodsky Construction over Algebraic Stacks}}

\author{Neeraj Deshmukh\thanks{Department of Mathematics, Boyd Research and Education Center, University of Georgia, Athens, GA 30602, USA.\\ 
\texttt{Email: neeraj.deshmukh@uga.edu}}
\and
Felix Sefzig\thanks{Institut f\"{u}r Mathematik, Universit\"{a}t Z\"{u}rich, Winterthurerstrasse 190, CH-8057 Z\"{u}rich, Switzerland.\\
\texttt{Email:felix.sefzig@math.uzh.ch}}}

\date{May 2025}

\input{marcosetup}

\input{theoremsetup}

\begin{document}

\maketitle

\begin{abstract}
In this article, we give a construction of the (un-)stable motivic homotopy category of an algebraic stack in the spirit of Morel-Voevodsky. We prove that this new construction agrees with the stable motivic homotopy category defined by Chowdhury \cite{Chowdhury_mot_stacks}. As an application, we prove a version of Bachmann's spectral rigidity theorem \cite{bachmann_rigidity} for algebraic stacks.
Moreover, we extend the construction of the framed motivic homotopy category
in \cite{mot_loop_spaces_21} to algebraic stacks, and prove the Reconstruction Theorem of \cite[Theorem 18]{Hoy_locthm21} in this setting. Finally, we discuss an extension of the formalism of cocomplete coefficient systems of \cite{DG_six_functor} to the setting of algebraic stacks.
\end{abstract}

\tableofcontents

\input{sh}

\bibliographystyle{alpha}
\bibliography{refs}

\end{document}

%% file: marcosetup.tex
\newcommand{\rom}[1]{\mathrm{#1}} 

\newcommand{\clg}[1]{\mathcal{#1}} 

\newcommand{\A}{\mathds{A}}

\newcommand{\G}{\mathbb{G}}

\newcommand{\Q}{\mathbb{Q}}

\newcommand{\Sp}{\mathrm{Sp}}
\newcommand{\Z}{\mathbb{Z}}

\newcommand{\PR}{\mathds{P}}

\newcommand{\Unit}{\mathbb{1}}

\newcommand{\cart}{\mathrm{cart}}

\newcommand{\op}{\mathrm{op}}

\newcommand{\stb}{\mathrm{st}}

\newcommand{\mot}{\mathrm{mot}}
\newcommand{\Nis}{\mathrm{Nis}}

\newcommand{\et}{\mathrm{\acute{e}t}}

\newcommand{\CAlg}{\mathrm{CAlg}}

\newcommand{\fsyn}{\mathrm{fsyn}}
\newcommand{\fr}{\mathrm{fr}}

\newcommand{\PSh}{\mathrm{PSh}}
\newcommand{\Sh}{\mathrm{Shv}}

\newcommand{\id}{\mathrm{id}}

\newcommand{\Corrs}{\mathrm{Corr}}

\newcommand{\DA}{\mathbf{DA}}
  
\newcommand{\DM}{\mathbf{DM}}
\newcommand{\SH}{\mathbf{SH}}

\newcommand{\Corr}{\mathbf{Corr}}
\newcommand{\Sm}{\mathbf{Sm}}
\newcommand{\Cat}{\mathrm{Cat}}
\newcommand{\Sch}{\mathbf{Sch}}
\newcommand{\Stk}{\mathbf{Stk}}

\newcommand{\Hm}{\mathbf{H}}
\newcommand{\Prc}{\mathbf{Pr}}

\newcommand{\Mod}{\mathbf{Mod}}

\newcommand{\COSY}{\mathbf{CoSy}}
\newcommand{\STCOSY}{\mathbf{StCoSy}}

\newcommand{\CC}{\mathcal{C}}

\newcommand{\XC}{\mathscr{X}}
\newcommand{\YC}{\mathscr{Y}}

\newcommand{\notes}[1]{}

\makeatletter
\newcommand*{\triplerightarrow}[1]{\mathrel{
  \settowidth{\@tempdima}{$\scriptstyle#1$}
  \mathop{\vcenter{
    \offinterlineskip\ialign{\hbox to\dimexpr\@tempdima+1em{##}\cr
    \rightarrowfill\cr\noalign{\kern.5ex}
    \rightarrowfill\cr\noalign{\kern.5ex}
    \rightarrowfill\cr}}}\limits^{\!#1}}}
\makeatother

\renewcommand{\epsilon}{\ensuremath\varepsilon}

\renewcommand{\phi}{\ensuremath{\varphi}}

\DeclareMathOperator{\Hom}{Hom}

\DeclareMathOperator{\Fun}{Fun}

\DeclareMathOperator{\Map}{Map}
\DeclareMathOperator{\colim}{colim}

\DeclareMathOperator{\Spec}{Spec}

%% file: theoremsetup.tex
\newtheorem{theorem}{Theorem}[section]
\newtheorem{cor}[theorem]{Corollary}
\newtheorem{lemma}[theorem]{Lemma}
\newtheorem{prop}[theorem]{Proposition}

\theoremstyle{definition}
\newtheorem{defn}[theorem]{Definition}

\theoremstyle{definition}
\newtheorem{remark}[theorem]{Remark}

\theoremstyle{definition}
\newtheorem{example}[theorem]{Example}

\newtheorem{construction}[theorem]{Construction}

\newtheorem*{propo*}[theorem]{Proposition}

%% file: sh.tex
\section{Introduction}

Let $X$ be a scheme. The Morel-Voevodsky stable motivic homotopy category over $X$, denoted $\SH(X)$, can be described as a suitable localization and stabilization of the category of sheaves of spectra on the Nisnevich site $(\Sm_X,\rm{Nis})$ of smooth schemes over $X$. This construction describes a functor on the category of schemes taking values in stable $\infty$-categories that satisfies Nisnevich descent.
Thus, given a Nisnevich local epimorphism $p:Y\rightarrow X$, we have an equivalence
\begin{equation}\label{equation-nis-descent}
p^*\colon \SH(X)\overset{\sim}{\rightarrow} \lim_\Delta \SH(Y_\bullet).
\end{equation}
Here, $Y_{\bullet}$ is the \v{C}ech nerve of $p \colon Y\rightarrow X$.
This statement also holds for other motivic categories (for example, $\SH^{\rm{fr}}(X)$, $\DM(X,\Lambda)$, $\DA(X,\Lambda)$, etc.). A low-brow way of thinking about the above $\infty$-categorical equivalence is the following. For any cohomology theory $E^*$ on $X$, there exists a descent spectral sequence converging to $E^*(X)$ whose $E_1$-page is described by the pullback of $E^*$ to the $Y_i$'s.

This paper aims to describe $\SH$ (and other motivic categories) over an algebraic stack in the spirit of Morel-Voevodsky. More specifically, we define $\SH(\XC)$ as a suitable localization and stabilization of the category of sheaves of spectra on the smooth-Nisnevich site of the algebraic stack $\XC$ (see Construction \ref{constr_SH}). Our choice of the smooth-Nisnevich site is motivated in part by the theory of smooth-\'etale sheaves on $\XC$, and in part by the direct analogy of this approach with the case of schemes.

In \cite{khan_ravi_cohomology} and \cite{Chowdhury_mot_stacks} the authors describe two categories $\SH_{\triangleleft}(\XC)$, and $\SH_{\rom{ext}}(\XC)$. These categories are equivalent models of the stable motivic homotopy category over an algebraic stack $\XC$, as observed in \cite{Chowdhury_mot_stacks} and \cite{deshmukh2023motivichomotopytypealgebraic}.
In fact, one may view these constructions as models of the right Kan extension of $\SH$ from schemes to algebraic stacks.
This approach should be thought of as the right-hand side of Equivalence (\ref{equation-nis-descent}).
On the other hand, the construction in Section \ref{section-morel-voevodsky-construction} belongs to the left-hand side. In fact, we will show that Equivalence (\ref{equation-nis-descent}) continues to hold for an algebraic stack, and we have equivalences (see Theorem \ref{thm-sh}),
\begin{equation}\label{equation-lisse-extension-kan-extension-morel-voevodsky}
    \SH(\XC)\overset{\sim}{\rightarrow} \SH_{\triangleleft}(\XC)\overset{\sim}{\rightarrow} \SH_{\rom{ext}}(\XC).
\end{equation}
The above equivalences also hold for other motivic categories as described in the text.

In \cite{chowdhury2024}, the authors define the category $\SH_{\rm{cl}}(\XC)$ and show that it is equivalent to $\SH_{\rm{ext}}(\XC)$ (\cite[Theorem 3.2.1]{chowdhury2024}). The description of the category $\SH_{\rom{cl}}(\XC)$ is very closely related to Construction \ref{constr_SH} (see Remark \ref{remark-comparison-chowdhury-dangelo} and also \cite[Remark 3.23]{chowdhury2024}). Morally, $\SH_{\rom{cl}}(\XC)$ also belongs to the left-hand side of Equivalence (\ref{equation-nis-descent}).

In addition to its aesthetic value, Construction \ref{constr_SH} of $\SH(\XC)$ also offers some conceptual advantages over the right Kan extension $j_*\SH(\XC)$. For example, one can easily identify a family of generators for $\SH(\XC)$ (Proposition \ref{proposition-generators-SH}). Moreover, the construction of $f_\#$ for smooth (possibly non-representable) morphisms $f\colon \XC\rightarrow\YC$ of algebraic stacks becomes almost obvious.
This is rather tricky for $j_*\SH(\XC)$, and can get quite involved (see Lemma \ref{lemma-extending-elements-cosy-to-stcosy}, for example, or \cite[Remark 3.26]{chowdhury2024}).
A final example we would like to mention concerns the framed version $\SH^\fr(\XC)$ (see Section \ref{section-framed}).
The constant sheaf $\Z_\XC$ on the smooth-Nisnevich site of $\XC$ has a natural notion of framed transfers. In particular, its framed suspension $\Sigma^\infty_{T,\fr}\Z_\XC$ is an object of $\SH^\fr(\XC)$ which describes the motivic cohomology spectrum $H\Z_\XC$ (see Corollary \ref{corollary-framed-motivic-cohomology-spectrum}) in the sense of Spitzweck \cite{Spitzweck_P1_spectrum}.

The key notion required to establish the chain of equivalences in (\ref{equation-lisse-extension-kan-extension-morel-voevodsky}) above is that of a smooth-Nisnevich covering. 
 
\begin{defn}\label{defn_sm_nis}
    A morphism of algebraic stacks $p \colon X \to \XC$ is a \textit{smooth Nisnevich covering} if it is smooth and, for every field $K$, every morphism $\Spec k \to \XC$ admits a lift along $p$.
\end{defn}

\begin{example}
Let $p \colon \XC \to S$ be a smooth morphism of stacks, where $S$ is a quasi-separated algebraic space. Then
every point $y \in  |S|$ has a residue field $k(y)$ \cite[\href{https://stacks.math.columbia.edu/tag/03JV}{Tag 03JV}]{stacks-project}. Hence, $p$ is a smooth Nisnevich covering if and
only if for every $y \in  |S|$, the morphism $\XC \times_S \Spec k(y) \to \Spec k(y)$ admits a section. In particular, if $\XC$ is
also an algebraic space, then \'etale smooth Nisnevich coverings are equivalent to Nisnevich covers in the sense of
\cite[Section 3]{HR_etale_devissage}.
\end{example}

In \cite{deshmukh2023motivichomotopytypealgebraic} and \cite{Chowdhury_mot_stacks}, the second equivalence in (\ref{equation-lisse-extension-kan-extension-morel-voevodsky}) above is proved using smooth-Nisnevich coverings\footnote{Strictly speaking, \cite{Chowdhury_mot_stacks} proves this for algebraic stacks admitting \textit{Nisnevich local sections}. But by \cite[Theorem A.1]{deshmukh2023motivichomotopytypealgebraic} or \cite[Proposition 2.4]{chowdhury2024}, all algebraic stacks admit Nisnevich local sections.}. In this text, we prove the first one. The following theorem implies that every algebraic stack admits a Nisnevich local epimorphism from a scheme (see also \cite[Proposition 2.4]{chowdhury2024}). 

\begin{theorem}{(\cite[Theorem A.1]{deshmukh2023motivichomotopytypealgebraic})}\label{theorem_hall}
    Let $\XC$ be an algebraic stack. There exists a smooth Nisnevich covering $p \colon  X \to \XC$, where $X$ is a scheme. If $\XC$ is quasi-compact and quasi-separated with affine stabilizers, then we may take $X$ to be affine. Moreover, if $\XC$ is Deligne-Mumford, then we can take $p$ to be \'etale.
\end{theorem}

The fundamental idea used in all the descent arguments in this paper is the following.
A smooth-Nisnevich cover of an algebraic stack $\XC$ gives rise to a Nisnevich local resolution of $\XC$ by a simplicial scheme.

\begin{prop}\label{prop_cechnerve_niscover}
    Let $\XC$ be an algebraic stack and let $p \colon X \to \XC$ be a smooth Nisnevich cover by a scheme. Then the associated \v{C}ech nerve gives rise to a Nisnevich equivalence $|X_\bullet| \to \XC$ in $\PSh(\Sm_\XC)$.
\end{prop}
\begin{proof}
    \cite[Proposition 3.1]{deshmukh2023motivichomotopytypealgebraic}.
\end{proof}

\subsection{Outline}

The text of this paper is organized as follows. In Section \ref{section-morel-voevodsky-construction}, we describe the construction of $\SH(\XC)$ and $\SH_{\et}(\XC)$ for an algebraic stack and prove that they satisfy the appropriate notion of descent (Theorem \ref{thm-sh}).
The \'{e}tale story leads us into a discussion about hypercomplete \'{e}tale sheaves.
In particular, we discuss the rigidity theorem of \cite{bachmann_rigidity} for stacks (Theorem \ref{theorem-rigidity}).
We also prove a version of Ayoub's equivalence between \'{e}tale motives with and without transfers (Theorem \ref{thm-transfer-with-or-without}) in the context of regular algebraic stacks over a field.

In Section \ref{section-framed}, we describe a model for the category of framed motivic spectra $\SH^{\fr}$ over $\XC$ and prove that it satisfies Nisnevich descent. Using the results of \cite{Hoy_locthm21}, this gives an equivalence between the framed and unframed versions of the stable motivic homotopy categories (Corollary \ref{corollary-framed-SH-limit}).
These observations result in an explicit description of the motivic cohomology spectrum over $\XC$ as the framed suspension of the (framed) constant sheaf $\Z$ (Corollary \ref{corollary-framed-motivic-cohomology-spectrum}).

In Section \ref{section-coarse-space-morphism}, we analyze the coarse space morphism of a tame Deligne-Mumford stack and show that it induces a fully faithful functor on the category of motives with $\Q$-coefficients (Theorem \ref{theorem-coarse-space-map-fully faithful}).

In Section \ref{section-coefficient-systems}, we discuss the theory of cocomplete coefficient systems as in \cite{DG_six_functor} in the context of algebraic stacks. We show that $\SH$ continues to be the universal six functor formalism on $\Stk_S$ (Corollary \ref{corollary-SH-initial}). We conclude by observing that Ayoub's \'{e}tale and $I$-adic realization functors extended to algebraic stacks continue to commute with the six operations.

\subsection{Acknowledgements}
We want to thank Joseph Ayoub for helpful discussions and suggestions at various stages of this work. We would like to thank Piotr Achinger for helpful discussions on the subject of this paper, in particular, about the proof of Proposition \ref{prop_nis_sheaves}. We would also like to thank Chirantan Chowdhury and Alessandro D'Angelo for discussions about motives and stacks and their paper \cite{chowdhury2024}. Many statements in Section \ref{section-coefficient-systems} have been inspired by these discussions. The first author also thanks Fr\'{e}d\'{e}ric D\'{e}glise for a helpful email correspondence about correspondences, transfers, and the contents of \cite{cisinski_deglise_mm}.

The first author was supported by the project
KAPIBARA, funded by the European Research Council (ERC) under the European Union's Horizon
2020 research and innovation programme (grant agreement No 802787). 

\subsection{Notation and conventions}
Throughout this article, we will work exclusively within the framework of infinity categories.
In particular, we adopt the following conventions and notations.
We write $\clg{S}$ for the $\infty$-category of spaces and $\Sp$ for the $\infty$-category of $S^1$-spectra.
For an $\infty$-category $C$, we write $\PSh(C) := \mathrm{Fun}(C,\clg{S})$ for the $\infty$-category of presheaves of spaces on $C$ and $\PSh(C)_*$ for the category of pointed presheaves of spaces.
If $(C,\tau)$ is a site, we write $\Sh_\tau(C)$ for the $\infty$-category of sheaves of spaces on $(C,\tau)$. If $\clg{T}$ is an $\infty$-topos we write $\clg{T}^\wedge$ for the hypercompletion of $\clg{T}$. A morphism of sites $f \colon (\clg{C},\tau) \to (\clg{D},\sigma)$ is given by a continuous functor $f \colon \clg{D} \to \clg{C}$ such that the corresponding pullback functor $f^* \colon \Sh_\sigma(\clg{D}) \to \Sh_\tau(\clg{C})$ is exact.

For a scheme $X$, we write $X_\et$ for the small \'etale site of $X$.
We write $\SH(X)$ for the stable motivic homotopy category and $\otimes$ for the symmetric monoidal structure on $\Hm(X)_*$ and $\SH(X)$.

\section{The Morel-Voevodsky construction over algebraic stacks}\label{section-morel-voevodsky-construction}
We are going to present an alternative definition, in the spirit of Morel-Voevodsky, of the stable motivic homotopy category of an algebraic stack. 
In addition, we prove that this new definition agrees with the one given in \cite{Chowdhury_mot_stacks}, which was originally given for Nisnevich-local stacks; however, by Theorem \ref{theorem_hall}, all algebraic stacks satisfy this assumption.

The stable motivic homotopy category $\SH(X)$ of a scheme $X$ is defined as the 
$\PR^1$-stabilization of $\mathbf{H}(X)_*$, the category of pointed $\mathbb{A}^1$-invariant Nisnevich sheaves of spaces on $\Sm_X$. Moreover, 
the assignment $X \mapsto \SH(X)$, $f \mapsto f^*$ defines a functor of presentable stable $\infty$-categories
\[
\SH \colon \Sch^\op \to \mathrm{Pr}^L_{\stb}.
\]
In fact, $\SH(X)$ has the structure of a symmetric monoidal presentable stable $\infty$-category, and the pullback functors are symmetric monoidal; in other words, $\SH$ defines a functor
\[
\SH \colon \Sch^\op \to \mathrm{\CAlg}(\mathrm{Pr}^L).
\]
Chowdhury defined $\SH$ for Nisnevich local stacks as the right Kan extension $j_* \SH$ of 
$\SH$ along the inclusion $j\colon \Sch \to \Stk$ of schemes into Nisnevich-local Stacks, which by Theorem \ref{theorem_hall} is the category of all algebraic stacks. 
We write $j_*\SH(\XC)$ for the right Kan extension and $\SH(\XC)$ the Morel-Voevodsky version, see construction below.

We begin by defining the smooth Nisnevich and smooth \'etale site of an algebraic stack. For this, we view the $\infty$-category of stacks as a subcategory of $\PSh(\Sch,\clg{S})$. 

\begin{defn}
    Let $\XC$ be an algebraic stack. 
    \begin{enumerate}
    \item Denote by $\Stk$ the $\infty$-category of stacks. The Nisnevich, respectively the \'etale topology, on $\Stk$ is the topology generated by families of maps $\{U_i \to \XC\}$, 
    such that $\bigsqcup_i U_i \to \XC$ is a Nisnevich, respectively an \'etale, local epimorphism of presheaves of spaces.
    \item Denote by $\mathbf{SmRep}_\XC$ the 1-category of stacks that are smooth and representable over $\XC$. Moreover, denote by $\Sm_\XC$ the category of smooth schemes over $\XC$. 
    \item The Nisnevich, respectively the \'etale topology, on $\mathbf{SmRep}_\XC$ or on $\Sm_\XC$ is the restriction of the Nisnevich topology, respectively of the \'etale topology, along the inclusion $\mathbf{SmRep}_\XC \to \Stk$ or $\Sm_\XC \to \Sch$.
    \end{enumerate}
\end{defn}

\begin{remark}
A family of maps of schemes is a Nisnevich cover in the usual sense if and only if it is a Nisnevich cover of Stacks. Moreover, a smooth Nisnevich covering $X \to \XC$ as in Definition \ref{defn_sm_nis} is a Nisnevich cover in the above sense.
\end{remark}

\begin{construction}[Morel-Voevodsky]\label{constr_SH}
Let $\XC$ be an algebraic stack over $S$.
We define $\mathbf{H}(\XC)$ as the localization of $\PSh(\mathbf{SmRep}_\XC)$ at the class of Nisnevich locally $\infty$-connective morphisms and the class of maps
\[
\left\{\mathrm{pr}_1 \colon U \times_\XC \A^1_\XC \to U \mid U \in \mathbf{SmRep}_\XC \right\}.
\]
The corresponding localization functor is called motivic localization
\[
L_\mot \colon \PSh(\mathbf{SmRep}_\XC) \to \Hm(\XC),
\]
and admits a fully faithful right adjoint. Consequently, we may view $\Hm(\XC)$ as the full subcategory of $\PSh(\mathbf{SmRep}_\XC)$ of $\A^1$-invariant Nisnevich sheaves.
Finally, we define $\SH(\XC)$ as the
$(\PR^1_\XC,\infty)$-stabilization, with respect to the usual wedge product, $\otimes$, of the pointed category $\Hm (\XC)_*$, see for example \cite[Section 2]{robalo15} for more details on the stabilization process. 
As for schemes, this construction defines a functor
\begin{equation*}
    \SH \colon \Stk^\op \to \mathrm{\CAlg}(\mathrm{Pr}^L),
\end{equation*}
which, by definition, agrees with $\SH$ on schemes over $S$. In particular, there
exists a natural transformation $u\colon  \SH \to j_* \SH$.
\end{construction}

\begin{construction}[\'Etale version]
The stable \'etale motivic homotopy category is constructed exactly as its Nisnevich counterpart, except that we invert all \'etale locally $\infty$-connective morphisms instead of the Nisnevich local equivalences. We denote the resulting categories by $\Hm_\et(\XC)$ and $\SH_\et(\XC)$. In particular, the \'etale motivic homotopy category is obtained as a localization of the hypercomplete 
\'etale topos on $\mathbf{SmRep}_\XC$
\[
L_\mot \colon \Sh_\et(\mathbf{SmRep}_\XC)^\wedge \to \Hm_\et(\XC).
\]
\end{construction}

\begin{remark}
    We use the smooth representable site of $\XC$ in the above constructions to ensure that $(\SH,f^*)$ is clearly functorial for all morphisms of stacks. On the other hand, as the next proposition shows, we can equivalently define $\SH(\XC)$ using the smooth Nisnevich site. Moreover, if $f \colon \XC \to \YC$ is a representable morphism, then the functor
    \[
    f^* \colon \SH_\tau(\YC) \to \SH_\tau(\XC)
    \]
    is induced by the functor $\Sm_\YC \to \Sm_\XC$ given by $U \mapsto U \times_\YC \XC$. 
\end{remark}

\begin{prop}
    Let $\XC$ be an algebraic stack. The inclusion functor $u \colon \Sm_\XC \to \mathbf{SmRep}_\XC$ induces an equivalence of topoi
    \[
    \Sh_\tau(\Sm_\XC) \xrightarrow{\sim} \Sh_\tau(\mathbf{SmRep}_\XC), 
    \]
    with $\tau \in \{\Nis,\et\}$.
    This equivalence identifies $\Hm(\XC)$ with the full subcategory of $\A^1$-invariant sheaves on $\Sm_\XC$.
\end{prop}
\begin{proof}
Since we are dealing with sheaves on $1$-categories, by \cite[Lemma 6.4.5.6]{HTT09}, it suffices to check that the ordinary sheaf topoi agree. For this, we check the condition from \cite[Tag 039Z]{stacks-project}. Since $u$ is fully faithful and clearly continuous, it suffices to check that $u$ is cocontinuous and every $V \in \mathbf{SmRep}_\XC$ admits a Nisnevich covering by smooth schemes over $\XC$. The second claim follows directly from Theorem \ref{theorem_hall}. For the first claim, let $U \in \Sm_\XC$ and let $\clg{V} := \{V_i \to u(U)\}$ be a Nisnevich cover in $\mathbf{SmRep}_\XC$.
For all $i$, let $U_i \to V_i$ be a smooth Nisnevich cover by a scheme, which exists again by Theorem \ref{theorem_hall}, then $\{U_i \to U\}$ is a Nisnevich covering in $\Sm_\XC$ and $\{u(U_i) \to u(U)\}$ refines $\clg{V}$.
\end{proof}

The main theorem of this section is that the Morel-Voevodsky construction of $\SH$ agrees with the right Kan extension of $\SH$ along the inclusion $j \colon \Sch \to \Stk$.

\begin{theorem}\label{thm-sh}
Let $\XC$ be an algebraic stack. The natural functors
\[
u \colon  \SH(\XC) \to  j_* \SH(\XC)
\]
and
\[
u \colon  \SH_\et(\XC) \to  j_* \SH_\et(\XC)
\]
are equivalences of symmetric monoidal presentable stable $\infty$-categories.
\end{theorem}

\begin{remark}\label{remark-comparison-chowdhury-dangelo}
    Chowdhury and D'Angelo independently gave an equivalent construction of $\SH(\XC)$ and proved a version of Theorem \ref{thm-sh}, \cite[Theorem 3.21]{chowdhury2024}.
    Concretely, they work with the category of all Artin Stacks endowed with the so-called NL-topology. A cover for the NL-topology is a morphism $\XC \to \YC$ that admits sections after pullback along a smooth Nisnevich cover of $\YC$ by a scheme. Finally, that the two constructions agree is an application of Theorem \ref{theorem_hall} and the fact that for smooth representable morphisms, the notion of NL-cover and Nisnevich cover agree.
\end{remark}

Before we discuss the proof of the main theorem, let us collect some standard facts about the categories $\Hm(\XC)$ and $\SH(\XC)$. The analogous results also hold in the \'etale setting. In particular, it follows that $\SH(\XC)$ is generated by suspension spectra of representable sheaves, and that $f^*$ admits a left adjoint $f_\#$ even for not necessarily representable, smooth morphisms $f$.

\begin{lemma}\label{lem_SH_descriptions}
    Let $\XC$ be an algebraic stack. The natural map $\A^1_\XC \to *$ is a homotopy equivalence in $\Hm(\XC)$. Moreover, the natural maps 
    \begin{equation*}
        (\A^n/(\A^n \smallsetminus 0),1) \to \Sigma^n (\G_m,1)^{\otimes n}, \quad (\G_m,1)\otimes S^1 \to (\PR^1,\infty)
    \end{equation*}
    are homotopy equivalences in $\Hm(\XC)_*$. Consequently, there are natural equivalences of stable $\infty$-categories
    \begin{equation*}
        \SH(\XC) \simeq \Hm(\XC)_*[(\PR^1_\XC)^{-1}] \simeq \Hm(\XC)_*[(\G_m \otimes S^1)^{-1}] \simeq \SH^{S^1}(\XC)[\G_m^{-1}],
    \end{equation*}
    where $\SH^{S^1}(\XC) := \Hm(\XC)_*[(S^1)^{-1}]$.
\end{lemma}

\begin{lemma}
    Let $\XC$ be an algebraic stack and let $\clg{E}$ be a vector bundle over $\XC$. 
    The canonical morphism of pointed sheaves 
    \[
    \PR(\clg{E}\oplus \clg{O}_\XC)/\PR(\clg{E}) \to \mathrm{Th}(\clg{E})
    \]
    is an $\A^1$-equivalence.    
\end{lemma}
\begin{proof}
    The proof works as in the scheme case \cite[Proposition 2.17]{MV_A1_homotopy}. 
    The Zariski cover of $ \PR(\clg{E}\oplus \clg{O}_\XC)$ given by
    \[
     \PR(\clg{E}\oplus \clg{O}_\XC) = \clg{E} \cup  (\PR(\clg{E}\oplus \clg{O}_\XC)\smallsetminus \XC)
    \]
    where the closed embedding of $\XC$ into $\PR(\clg{E}\oplus \clg{O}_\XC)$ is the composition of the embedding of $\clg{E}$ with the zero section.
    The above cover induces an isomorphism of pointed sheaves
    \[
    \mathrm{Th}(\clg{E}) = \PR(\clg{E}\oplus \clg{O}_\XC)/(\PR(\clg{E}\oplus \clg{O}_\XC)\smallsetminus \XC).
    \]
    Thus, it suffices to show that the embedding $\PR(\clg{E}) \to \PR(\clg{E}\oplus \clg{O}_\XC)\smallsetminus \XC$ is an $\A^1$-equivalence. 
    By Theorem \ref{theorem_hall}, there exists a smooth Nisnevich cover of $\XC$ by a scheme $p \colon X \to \XC$. Since we have a Nisnevich equivalence $|\check{C}^\bullet(p)| \sim \XC$ it suffices to show the claim after base change to $X_n$ for all $n$.
    Because the formation of the projective space commutes with base change,
    we are reduced to the corresponding statement in the case of a scheme, which is \cite[Proposition 2.17]{MV_A1_homotopy}.
\end{proof}

\begin{lemma}\label{lem_cyclic_perm}
    Let $T = \G_m \otimes S^1$. The cyclic permutation $T \otimes T \otimes T$ is a homotopy equivalence in $\Hm(\XC)_*$.
\end{lemma}
\begin{proof}
    By the same argument as in the previous proof, we reduce to the scheme case, where the lemma is well-known.
\end{proof}
The previous two lemmas imply that $\SH(X)$ is
computed by either the colimit in $\Pr^L$ of 
\[
\Hm(\XC)_* \xrightarrow{\Sigma_T} \Hm(\XC)_*\xrightarrow{\Sigma_T} \Hm(\XC)_* \xrightarrow{\Sigma_T} \cdots, 
\]
where $T = (\G_m \otimes S^1,1) \simeq (\PR^1_\XC, \infty)$ or by the limit in $\Pr^R$ of
\[
\cdots \xrightarrow{\Omega_T} \Hm(\XC)_* \xrightarrow{\Omega_{T}} \Hm(\XC)_* \xrightarrow{\Omega_{T}} \Hm(\XC)_*.
\]

\begin{prop}\label{proposition-generators-SH}
    Let $\XC$ be an algebraic stack. The $\infty$-categories $\Hm(\XC)$, $\Hm(\XC)_*$ are presentable and generated under colimits by the family $X$ (resp. $X_+$) for $X \in \Sm_\XC$. Moreover, the category 
    $\SH(\XC)$ is generated under colimits, desuspensions, and negative $\PR^1$-suspensions by the family $\Sigma^\infty_{\PR^1}X_+$ for $X \in \Sm_\XC$.
\end{prop}

\begin{proof}
    The presentability of all categories follows directly from the fact that $\PSh(\Sm_\XC)$ is presentable, and we are localizing with respect to a small class of morphisms. Moreover, since $\PSh(\Sm_\XC)$ is generated under the family $X \in \Sm_\XC$, so is $\Hm(\XC)$ and similarly for the pointed version. The fact that $\SH(\XC)$ is generated by $\Sigma^\infty_{\PR^1}X_+$ is a general fact about stabilization of symmetric monoidal $\infty$-categories \cite[Proposition 2.7.1]{binda2024logarithmic}.
\end{proof}

\begin{cor}\label{cor_a_sharp}
    Let $f \colon \XC \to \YC$ be a smooth morphism of Nisnevich local stacks. Then the functor $f^* \colon \SH(\YC) \to \SH(\XC)$ admits a left-adjoint.
\end{cor}

\begin{proof}
    Since $f$ is smooth, the forgetful functor induces a morphism of sites
    \[
    f \colon (\mathbf{SmRep}_{\YC},\Nis) \to (\mathbf{SmRep}_\XC,\Nis).
    \]
    By \cite[Theorem 4.5.10]{ayoubthesisII}, $f$ gives rise to a pair of adjoint functors
    \[
    f_\# \colon  \Sh_\Nis(\mathbf{SmRep}_\XC,\CC) \leftrightarrows \Sh_\Nis(\mathbf{SmRep}_\YC,\CC) :f^*,
    \]
    where $\CC$ denotes either the category of spaces or the category of spectra.
    Since $f^*$ preserves $\A^1$-invariant sheaves, $f$ induces a functor
    \[
    f^*\colon \SH(\XC)^{S^1} \to \SH(\YC)^{S^1}
    \]
    with left adjoint given by $L_\mot f_\#$.    
    Finally, by the description of $\SH(\XC)$ in Lemma \ref{lem_SH_descriptions}, the smooth projection formula\footnote{The proof of the smooth projection formula reduces to the representable case, where it follows from an easy computation.}, and the fact that $f^* \G_{m,\YC} = \G_{m,\XC}$, it follows that
    $f_\#$ extends uniquely to a functor
    \[
    f_\# \colon \SH(\XC) \to \SH(\YC).
    \]
    In particular, $f_\#$ is determined by
    \[
    f_\#(\Sigma_{\PR^1}^\infty X_+) = \Sigma^\infty_{\PR^1} f_\#(X_+)
    \]
    for all $X \in \Sm_\XC$.
\end{proof}

The main theorem is a consequence of the following descent property of sheaves on the smooth Nisnevich (or \'etale) site applied to the \v{C}ech nerve of a smooth Nisnevich cover.
Throughout this section, we write $\tau$ for the \'etale or Nisnevich topology.

\begin{prop}\label{prop_nis_sheaves}
    Let $X_\bullet \in \Sm_\XC$ be a simplicial diagram. Write $p_n$ for the structure map of $X_n$.
    Assume that $|X_\bullet| \to \XC$ is a $\tau$-local equivalence. 
    Then, the functors $p_n^* \colon \Sh_\tau(\Sm_\XC) \to \Sh_\tau(\Sm_{X_n})$ induce an equivalence in $\Pr^L$
    \begin{equation*}
        p^* \colon \Sh_\tau(\Sm_\XC) \to \lim_{\Delta} \Sh_\tau(\Sm_{X_\bullet}).
    \end{equation*}
\end{prop}

Let $X_\bullet \in \Sm_\XC$ be a simplicial diagram. We begin by describing the $\infty$-category $\lim_{\Delta} \Sh_\tau(\Sm_{X_\bullet})$ and the functor $p^*$ more explicitly.
By \cite[6.3.2.3]{HTT09}, the limit $\lim_\Delta \Sh_\tau(\Sm_{X_\bullet})$ in $\Pr^L$ agrees with the limit in $\Cat_\infty$, and limits of $\infty$-categories are computed as follows, see \cite[Section 3.3]{HTT09} for more details. The cosimplicial diagram 
\begin{align*}
G \colon \Delta &\to \Cat_\infty \\ [n] &\mapsto \Sh_\tau(\Sm_{X_n})
\end{align*}
corresponds to a cocartesian fibration
\[
\clg{E} \colon \int_\Delta G \to \Delta,
\]
where $\int_\Delta G$ is the so-called $\infty$-category of elements of $G$, and the limit 
$\lim_\Delta G([n])$ is given by the category of cocartesian sections of $\clg{E}$. In our situation, the category
$\int_{\Delta} G$ is the (homotopy coherent nerve) of the simplicial category
with objects
\[
([n], F_n) \in \Delta \times \Sh_\tau(\Sm_{X_n})
\]
and morphisms $([m], F_n) \to ([n], F_n)$ given by pairs 
\[(\alpha,f_\alpha) \in \Hom_\Delta([m],[n]) \times \Hom_{\Sh(\Sm_{X_n})}(\alpha^*F_m, F_n),\]
where $\alpha^* \colon \Sh_\tau(\Sm_{X_m}) \to \Sh_\tau(\Sm_{X_n})$ is the pullback along the map
$X_n \to X_m$ corresponding to $\alpha$. The cocartesian fibration $\clg{E}$ is given by $([n],F_n) \mapsto [n]$.
In particular, the $\infty$-category 
of sections $\Fun_{\Delta}(\Delta, \int_\Delta G)$ can be described explicitly.

\begin{defn}
    Let $X_\bullet \in \Sm_\XC$ be a simplicial diagram. We define $\Sh_\tau(\Sm_{X_\bullet})$ as the $\infty$-category $\Fun_{\Delta}(\Delta, \int_\Delta G)$ of sections of $\clg{E}$     
    \footnote{We think of $\Sh_\tau(\Sm_{X_n})$ as the category of sheaves on the simplicial site $\Sm_{X_n}$, although this is not well-defined because the smooth site is not functorial.}. 
    Explicitly, the objects of $\Sh_\tau(\Sm_{X_\bullet})$ are collections of sheaves $F_n \in \Sh_\tau(\Sm_{X_n})$ for every $n \geq 0$, together with compatible morphisms
    \[
    F(\alpha) \colon \alpha^*F_m \to F_n
    \]
    for every $\alpha \colon [m] \to [n]$. We write $F = (F_n,F(\alpha)) \in \Sh_\tau(\Sm_{X_n})$ or simply $F = (F_n)$ if the morphisms $F(\alpha)$ are clear from the context.
\end{defn}

With this notation, the category of cocartesian sections of $\clg{E}$ is the full subcategory of $\Sh(\Sm_{X_\bullet})$
given by $F = (F_n,F(\alpha))$
for which all the morphisms $F(\alpha)$ are equivalences.

\begin{lemma}
\label{lem_comp_functor}
    Let $X_\bullet \in \Sm_\XC$ be a simplicial diagram. The functors $p_n^*$ induce a functor
    \begin{equation*}
        p^{*}\colon \Sh_\tau(\Sm_{\XC}) \to \Sh_\tau(\Sm_{X_\bullet})
    \end{equation*}
    with right adjoint
    \begin{equation*}
        p_*F = \lim_{\Delta} p_{n,*} F_n,
    \end{equation*}
    for $F = (F_n, F(\alpha)) \in \Sh_\tau(\Sm_{X_\bullet})$.
\end{lemma}

\begin{proof}
Let $\Sh_\tau(\Sm_{\XC_\bullet})$ be the category of 
cosimplicial sheaves on $\XC$.
We define $p^*$ as the composition
\[
\Sh_\tau(\Sm_\XC) \xrightarrow{\rom{const}} \Sh_\tau(\Sm_{\XC_\bullet})
\xrightarrow{p^*_{\bullet}} \Sh_\tau(\Sm_{X_\bullet}).
\]
By \cite[Chapter 18]{Hirsch03}, taking the homotopy limit of a cosimplicial sheaf and the functor taking a sheaf to the constant cosimplicial sheaf form a pair of adjoint functors,
\begin{equation*}
    \mathrm{const} \colon  \Sh_\tau(\Sm_{\XC}) \leftrightarrows \Sh_\tau(\Sm_{\XC_\bullet}) : \lim_\Delta.
\end{equation*}
Thus, $p^*$ admits a right adjoint $p_*$, which is given by the formula above.
\end{proof}

\begin{proof}[Proof of Proposition \ref{prop_nis_sheaves}]
We start by showing that the unit $\id\to p_* \circ p^*$ is an equivalence.
Let $F \in \Sh_\tau(\Sm_\XC)$ and $U \in \Sm_\XC$. We have the following natural equivalences
\begin{equation*}
    p_* p^* F (U) \simeq \lim_{\Delta} p_{n,*} p^*F \simeq \lim_{\Delta} F(U\times_\XC X_n) \simeq F(U),
\end{equation*}
because $U \times_\XC X_\bullet \to U$ is a $\tau$-equivalence. 

The essential surjectivity of $p^*$ follows from the general descent theory for $\infty$-topoi. 
First, note that the inclusion functors $i_n \colon \Sm_{X_n} \to \Sm_\XC/X_n$ induce fully faithful (see below) functors $i_n^* \colon \Sh_\tau(\Sm_{X_n}) \to \Sh_\tau((\Sm_\XC)/X_n) \simeq \Sh_\tau(\Sm_\XC)/X_n$, and hence we view $\Sh_\tau(\Sm_{X_\bullet})$ as a full subcategory of $\Sh_\tau(\Sm_{\XC})/X_\bullet$.
Our goal becomes to identify the essential image of $\Sh_\tau(\Sm_\XC)$ in this larger category. 
\footnote{We write $\Sh_\tau(\Sm_{\XC})/X_\bullet$ for the category of sections of the cocartesian fibration associated to $[n] \to \Sh_\tau(\Sm_\XC)/X_\bullet$.}
We have to following commutative diagram 
\begin{equation*}
\begin{tikzcd}
    & \Sh_\tau(\Sm_\XC) \ar[dr,"q^*"] \ar[dl,swap, "p^*"] \\
    \lim_\Delta \Sh_\tau(\Sm_{X_\bullet}) \ar[rr,hook,"\lim_\Delta i_\bullet"] \ar[d, hook] &&  \lim_\Delta \Sh_\tau(\Sm_\XC)/X_\bullet \ar[d,hook]\\
    \Sh_\tau(\Sm_{X_\bullet}) \ar[rr,hook,"i_\bullet"] && \Sh_\tau(\Sm_\XC)/X_\bullet
\end{tikzcd}
\end{equation*}
It follows from \cite[Proposition 6.3.5.14]{HTT09} that $q^*$ is an equivalence and hence $\lim_\Delta i_\bullet$ and $p^*$ must be equivalences.

It remains to show that the functors $i^*_n$ are fully faithful. Let $i_{n,*}$ be the right adjoints\footnote{Using that $i_n$ is continuous.}, given by
\begin{align*}
i_{n,*} \colon \Sh_\tau((\Sm_\XC)/X_n)  &\to \Sh_\tau(\Sm_{X_n}) \\
F &\mapsto F \circ i_n.
\end{align*}
We check that the unit map $F \to i_{n,*}i^*_n F$ is an equivalence.
Let $U \in \Sm_{X_n}$ and $F \in \Sh_\tau(\Sm_{X_n})$, then
\[
i_{n,*} i^*_n F (U) = i^*_n F(i_n(U)) = (\colim_{(V \in \Sm_{X_n}, i_n(U) \to i_n(V))^\op} F(V))^\# = F(U).
\]
\end{proof}

\begin{lemma}\label{lem_cart_sheaves_etale}
Let $\XC$ be an algebraic stack. Let $X_\bullet \in \Sm_\XC$ be a simplicial diagram such that $|X_\bullet| \to \XC$ is a \'etale-local equivalence. Then, the functor $p^*$ defined above restricts to an equivalence
\[
p^* \colon \Sh_\et(\XC)^\wedge \to \lim_\Delta \Sh_\et(\Sm_{X_\bullet})^\wedge.
\]
\end{lemma}
\begin{proof}
    Recall, a sheaf $F \in \Sh_\et(\XC)$ is called hypercomplete if it is local with respect to the class of $\infty$-connective morphisms, and the hypercomplete \'etale topos
    $\Sh_\et(\XC)^\wedge$ is naturally equivalent to the full subcategory of hypercomplete sheaves in $\Sh_\et(\XC)$.
    
    Fully faithfulness of $p^*$ follows exactly as in the proof of Proposition \ref{prop_nis_sheaves}. Since $p_{n,*}$ preserves hypercomplete sheaves \cite[Proposition 6.5.2.13]{HTT09}, it suffices to show that $p_n^*$ does. In other words, we need to show that $p^* F \in \lim_\Delta \Sh_\et(\Sm_{X_n})^\wedge \subset \lim_\Delta \Sh_\et(\Sm_{X_n})$ for every  $F \in \Sh_\et(\Sm_\XC)^\wedge$.
    Since hypercompletion commutes with taking the slice topos, \cite[Proposition 6.3.5.14]{HTT09} it follows that $i^*_\bullet p^* F \in (\Sh_\et(\Sm_\XC)/X_\bullet)$ is hypercomplete, with the notation as above.
    Since $i^*_\bullet$ is an equivalence on cartesian sections, it follows that $p^* F = i_{\bullet,*}i^*_\bullet p^* F$ is also hypercomplete.
\end{proof}

\begin{proof}[Proof of Theorem \ref{thm-sh}]
Let $\XC$ be a Nisnevich-local stack with smooth Nisnevich cover $X \to \XC$ with \v{C}ech cover $p \colon X_\bullet \to \XC$.
To simplify notation, we give the proof in the Nisnevich case. The arguments are the same in the \'etale case using Lemma \ref{lem_cart_sheaves_etale} to deal with hypercompleteness.

We have to show that
\begin{equation*}
    u \colon \SH (\XC) \to \lim_\Delta \SH(X_\bullet) \simeq j_*\SH(\XC)
\end{equation*}
is an equivalence in $\Prc^L$ because the forgetful functor $\mathrm{\CAlg}(\mathrm{Pr}^L) \to \Prc^L$ preserves limits, \cite[Corollary 3.2.2.5, Lemma 3.2.2.6]{lurie_HA}.
We start by showing that
the equivalence from Proposition \ref{prop_nis_sheaves}
\begin{equation*}
    p^* \colon \Sh_\Nis(\Sm_\XC) \leftrightarrows \lim_\Delta \Sh_\Nis(\Sm_{X_\bullet}): p_*
\end{equation*}
induces an equivalence 
\begin{equation}\label{eq_hm_descent}
\Hm(\XC)_* \simeq \lim_\Delta \Hm(X_\bullet)_*.
\end{equation}
Since the projection maps $p_n \colon X_n \to \XC$ are smooth, $p^*$ preserves $\A^1$-invariant sheaves, and hence it suffices to show that
\[
p_* \colon \lim_\Delta \Hm(X_\bullet) \to \Sh_\Nis(\XC)
\]
preserves $\A^1$-invariant sheaves.
Let $F = F_\bullet \in \lim_\Delta \mathbf{H}(X_\bullet)$ and let $U \in \Sm_\XC$, then 
we calculate
\begin{align*}
    p_*(F)(U \times_\XC \A^1_\XC) &= \lim_\Delta F_n(p_n^*(U\times_\XC \A^1_\XC)) \\
    & = \lim_\Delta F_n(U \times_\XC \A^1_\XC \times_\XC X_n) \\
    &= \lim_\Delta F_n((U \times_\XC X_n) \times_{X_n} (X_n \times _\XC \A^1_{\XC}) \\
    &= \lim_\Delta F_n(U \times_\XC X_n) = p_*F(U),
\end{align*}
where we used $\A^1$-invariance of the $F_n$'s in the last step.

Furthermore, each $p_n^*$ is symmetric monoidal and $p^*_n \PR^1_\XC = \PR^1_{X_n}$,
hence, $p^*$ extends to a symmetric monoidal functor
\begin{equation*}
    p^* \colon \SH(\XC) \to  \lim_\Delta \SH(X_\bullet).
\end{equation*}
In other words, we want to show that the functor
\[
\SH \colon \Stk^\op \to \Prc^L
\]
satisfies Nisnevich descent. It is enough to check this on the smooth site $\Sm_\XC$ for every stack $\XC$.
For the moment we denote by $\Hm(-)'_*$, $\SH(-)'$ the functors $\XC \mapsto \Hm(\XC)$, $\XC \mapsto \SH(\XC)$, and $f \mapsto f_*$.

By Lemma \ref{lem_cyclic_perm} and \cite[Corollary 2.22]{robalo15}, the presheaf $\SH(-)$ is computed as the pointwise colimit in $\Prc^L$ of
\[
\Hm(-)_* \xrightarrow{\Sigma_T} \Hm(-)_*\xrightarrow{\Sigma_T} \Hm(-)_* \xrightarrow{\Sigma_T} \cdots. 
\]
Since $\Prc^R \simeq (\Prc^L)^\op$, it follows from \cite[Corollary 5.5.3.4, Theorem 5.5.3.18]{HTT09} that the copresheaf $\SH'(-)$ can be computed
as the pointwise limit in $\Cat_\infty$ of the diagram
\[
\cdots \xrightarrow{\Omega_{\PR^1}} \Hm(-)'_* \xrightarrow{\Omega_{\PR^1}} \Hm(-)'_*  \xrightarrow{\Omega_{\PR^1}} \Hm(-)'_*.
\]
Since $\Omega_{\PR^1} \circ f^* = f^* \Omega_{\PR^1}$ for smooth morphisms (by the smooth projection formula), it follows that the presheaf $\SH(-)$ is computed as the limit in $\Cat_\infty$ of the diagram
\[
\cdots \xrightarrow{\Omega_{\PR^1}} \Hm(-)_* \xrightarrow{\Omega_{\PR^1}} \Hm(-)_*  \xrightarrow{\Omega_{\PR^1}} \Hm(-)_*.
\]
The theorem follows now from the fact that $\Hm(-)_*$ is a Nisnevich sheaf by (\ref{eq_hm_descent}).
\end{proof}

\begin{cor}[Coefficients]\label{et/coefficients-version}
    Let $X_\bullet \to \XC$ be as above. Let $\Lambda$ be a ring spectrum. Let $\tau$ be either the \'etale or Nisnevich topology. 
    Then the natural map
    \begin{equation*}
        a^* \colon \SH_\tau(\XC,\Lambda) \to \lim_\Delta \SH_\tau(X_\bullet,\Lambda)
    \end{equation*}
    is an equivalence of symmetric monoidal $\infty$-categories.
\end{cor}

\begin{proof}
To deal with coefficients, we use the following identification
\[
\SH_\tau(\XC,\Lambda) \simeq \mathbf{Mod}_{\underline{\Lambda}}(\SH_\tau(\XC)),
\]
where $\underline{\Lambda}$ denotes the constant sheaf with value $\Lambda$.
\end{proof}

\begin{defn}
    Let $\Lambda$ be a commutative ring. The category of motives (without transfers) over $\XC$ with coefficients in $\Lambda$
    \[\DA_{\tau}(\XC,\Lambda):=\mathbf{Mod}_{H\Lambda}(\SH_\tau(\XC)),\]
    where $H\Lambda$ is the Eilenberg-Maclane spectrum of $\Lambda$.
\end{defn}

\begin{cor}\label{corollary-descent-hz-mod}
    Let $\XC$ be a smooth algebraic stack over a field $k$ with smooth Nisnevich covering $X \to \XC$.
    Let $H\Z^\mot_\XC \in \SH(\XC)$ denote Spitzweck's motivic cohomology spectrum \cite{Spitzweck_P1_spectrum}. 
    Then
    \[
    \mathbf{Mod}_{H\Z^\mot_\XC} \simeq \lim_{\Delta}\mathbf{Mod}_{H\Z^\mot_{X_\bullet}},
    \]
    where $X_\bullet \to \XC$ denotes the \v{C}ech nerve of $X \to \XC$.
\end{cor}
\begin{proof}
    For any scheme or algebraic stack $\XC$, $H\Z^\mot_\XC:=p^* H\Z^\mot$ where $p\colon \XC\rightarrow\Spec \Z$ is the canonical map, and $H\Z^\mot\in \SH(\Z)$ is Spitzweck's motivic cohomology spectrum. Thus, Theorem \ref{thm-sh} applied to the full subcategory $\mathbf{Mod}_{H\Z^\mot_\XC}$ of $\SH(\XC)$ implies the statement.
\end{proof}

\begin{remark}
    Let $X$ be a regular scheme of finite Krull dimension over a field $k$ and $\Lambda$ a commutative ring in which the exponential characteristic of $k$ is invertible. By \cite[Theorem 3.1]{cisinski_deglise_integralmm}, we have an equivalence $\Mod_{H\Lambda_X}\simeq \DM(X,\Lambda)$ between the category of modules over $H\Lambda_X$ and Voevodsky triangulated category of $\Lambda$-motives with transfers over $X$. As transfers are poorly behaved over a general base, we will work with $\Mod_{H\Z^\mot_\XC}$-modules in the sequel.
\end{remark}

In this context, we have an adding transfer functor 
\[
a: \DA_{\et}(\XC,\Z)\rightarrow \Mod_{H\Z^{\mot,\et}_{\XC}}
\]
given by $- \otimes_{H\Z}H\Z^\mot_\XC$.
The following theorem is an extension of Ayoub's equivalence between $\DM_\et$ and $\DA_\et$ \cite[Theorem B.1]{ayoub_etreal} to regular algebraic stacks over a field.

\begin{theorem}[Motives with and without transfers]\label{thm-transfer-with-or-without}
    Let $\XC$ be a regular algebraic stack of finite Krull dimension over a field $k$ and $\Lambda$ be a commutative ring satisfying the conditions in \cite[Theorem B.1]{ayoub_etreal}. Then, adding transfers induces an equivalence
    \[
    a\colon  \DA_\et(\XC,\Lambda)\rightarrow \Mod_{H\Lambda^{\mot,\et}_{\XC}}.
    \]
\end{theorem}
\begin{proof}
    Let $X\rightarrow \XC$ be a smooth Nisnevich cover. As $\XC$ is regular, so is $X$. Let $X_{\bullet}\rightarrow\XC$ be the associated \v{C}ech nerve. As each $X_i$ is regular, by \cite[Theorem 3.1]{cisinski_deglise_integralmm} and \cite[Theorem B.1]{ayoub_etreal},
    \[a_i\colon  \DA_\et(X_i,\Lambda)\rightarrow \Mod_{H\Lambda^{\et}_{X_i}}\]
    is an equivalence for each $i$. 
    By Proposition \ref{et/coefficients-version} and Proposition \ref{corollary-descent-hz-mod}, both $\DA_\et$ and $H\Lambda^{\et}$ satisfy \'{e}tale descent and we are done.
\end{proof}

We conclude this section with a computation of the \'etale cohomology of an algebraic stack. Consider the pair of
adjoint functors on the motivic categories,
\[\epsilon^*\colon  \SH(\XC)\rightleftarrows \SH_\et(\XC)\colon  \epsilon_*\]
induced by the morphism of sites $\epsilon \colon  (\Sm_\XC, \et) \rightarrow(\Sm_\XC, {\Nis})$.

The following result says that the ``\'{e}tale sheafification" $\epsilon_*\epsilon^* H\Z^\mot_\XC$ of $H\Z^\mot_\XC$ describes \'{e}tale motivic cohomology.

\begin{cor}
    Let $\XC$ be a smooth algebraic stack over a field $k$. Let $m$ be an integer coprime to the characteristic of $k$. Then we have a canonical isomorphism
    \[H^p_{\rm \acute{e}t}(\XC, \mu_m^{\otimes q}) \simeq [\mathbf{1},\Sigma^{p,q}\epsilon_*\epsilon^* H\Z^\mot/m]_{\SH(\XC)} =:  H_\et^{p,q}(\XC,\Z/m).\]
    where the right hand side denotes \'{e}tale motivic cohomology.
\end{cor}

\subsection{Rigidity in \'etale motivic stable homotopy}

This subsection aims to generalize Bachmann's spectral rigidity theorem \cite{bachmann_rigidity} to algebraic stacks. The classical version of the theorem asserts that the canonical functor
\[
D(X_\et, \Z/n) \to \DM_\et(X, \Z/n)
\]
is an equivalence if $n \in \clg{O}(X)^\times$. The theorem was first proven in the case when $X$ is the spectrum of a field by Suslin-Voevodsky \cite[Theorem 4.4]{suslin_veov_singhom} \cite[Proposition 3.3.3]{Voev_2000}. Over a general base, the theorem was established by Ayoub for motives without transfers 
\cite[Theoreme 4.1]{ayoub_etreal} and for motives with transfers by Cisinski-Deglise \cite[Theorem 4.5.2]{cisinski_deglise_etalemotives}.
Furthermore, Bachmann proved the following spectral version of the theorem  \cite{bachmann_rigidity}, and \cite{bachmann_rigidity2}. Let $X$ be a scheme and let $\ell$ be a prime that is invertible on $X$. Consider the category of hypercomplete sheaves of spectra $(\Sh(X_\et),\Sp)^\wedge$ on the small \'etale site $X_\et$ of $X$. Then, the canonical functors
\[
\Sh(X_\et, \Sp)^\wedge \xrightarrow{\sigma} \SH^{S^1}_\et(X) \xrightarrow{\Sigma^{\infty}_{\G_m}} \SH_\et(X)
\]
become equivalences after $\ell$-completion (of a stable $\infty$-categories). The idea of the proof is to construct an invertible object (which is stable under pullbacks) $\Unit(1)_\ell \in (\Sh(X_\et, \Sp)^\wedge)^\wedge_\ell$ and an equivalence $\G_m \to \sigma (\Unit(1)_\ell)$ in $\SH^{S^1}_\et(X)^\wedge_\ell$.
Therefore, proving that $\sigma$ is an equivalence after $\ell$-completion also implies that $\Sigma^\infty_{\G_m}$ is an equivalence. With this at hand, we can reformulate and generalize the rigidity theorem as follows.

Since there is no good notion of a small \'etale site for a general algebraic stack, we consider the category of cartesian sheaves on the $\Sm_\XC$ as introduced in \cite{olsson_artin_stacks}. 

\begin{defn}
    A sheaf $F \in \Sh_\et(\Sm_\XC,\Sp)$ is \textit{cartesian} if for any morphism $f \colon U \to V$ in $\Sm_\XC$, the map of sheaves on $U_\et$ 
    \[
    f^{*} F_{V} \to F_{U}
    \]
    is an equivalence. We denote the full subcategory of cartesian sheaves by $\Sh_{\et,\cart}(\Sm_\XC,\Sp)$.
\end{defn}

If $\XC = X$ is a scheme (or a Deligne-Mumford stack), then 
the inclusion \[ i_X \colon X_\et \to \Sm_X\] induces an equivalence between
cartesian \'etale sheaves on $\Sm_X$ and sheaves on the small \'etale site of $X$. 
In particular, the functor $\Sh(X_\et,\Sp)^\wedge \to \SH^{S^1}_\et(X)$ from above factors through
the equivalence $i_X^* \colon \Sh(X_\et,\Sp)^\wedge \xrightarrow{\sim} \Sh_{\et,\cart}(\Sm_X,\Sp)^\wedge$.

\begin{theorem}[Rigidity]\label{theorem-rigidity}
    Let $\XC$ be an algebraic stack and let $\ell$ be a prime invertible on $\XC$. Let $F \in (\Sh_\et(\Sm_\XC, \Sp)^\wedge)^\wedge_\ell$ be an $\ell$-complete sheaf of spectra on $\Sm_\XC$. Then, the following are equivalent
    \begin{enumerate}
        \item $F$ is cartesian.
        \item $F$ is motivic.
    \end{enumerate}
    Moreover, $\G_m$ is $\otimes$-invertible in $\SH_\et^{S^1}(\XC)^\wedge_\ell$.
\end{theorem}

\begin{remark}
    In the scheme case, the above theorem is simply a direct reformulation of Bachmann's rigidity theorem.
\end{remark}

\begin{proof}
Let $F \in (\Sh_\et(\Sm_\XC, \Sp)^\wedge)^\wedge_\ell$ be an $\ell$-complete sheaf of spectra on $\Sm_\XC$. Let $p \colon X \to \XC$ be a smooth cover with \v{C}ech nerve $X_\bullet \to \XC$.
Assume that $F$ is cartesian and we want to show that $F$ is motivic. 
By Theorem \ref{thm-sh}, it suffices to show that the pullback to $X_\bullet$ is motivic.
Cartesian sheaves are clearly stable under pullbacks, hence the pullback of $F$ to $X_\bullet$ is cartesian and the claim follows from rigidity for schemes.

Conversely, assume that $F$ is motivic and we want to show that $F$ is cartesian. Since $F$ is hypercomplete, it suffices to show that the homotopy sheaves of $F$ are cartesian. Thus, we may assume that $F$ is a sheaf of groups. In this case, \cite[4.4,4.7]{olsson_artin_stacks} implies that $F$ is cartesian if and only if its pullback to $X_\bullet$ is cartesian. Thus, the claim follows again from rigidity for schemes.

Finally, we need to show that $\G_m$ is $\otimes$-invertible in $\SH^{S^1}_\et(\XC)^\wedge_\ell$. Again, by Theorem \ref{thm-sh}, we may check this after pulling back along $X_\bullet \to X$, where we know that $\G_m$ is invertible.
\end{proof}

We implicitly used that $\ell$-completion of stable $\infty$-categories commutes with homotopy limits in $\mathrm{CAlg}(\Pr^L)$ and with $\G_m$-stabilization \cite[Lemma 26]{bachmann_real_etale}.

\begin{lemma}\label{lem_pcompl_lim}
    Let $\CC \colon K \to \mathrm{\CAlg}(\mathrm{Pr}^L)$ be a simplicial diagram of presentable stable symmetric monoidal $\infty$-categories. There exists a natural equivalence
    \[
    (\lim_{k\in K} \CC(k))^\wedge_\ell \to \lim_{k\in K} \CC(k)^\wedge_\ell
    \]
    of symmetric monoidal $\infty$-categories.
\end{lemma}
\begin{proof}
    Follows from the limit description $\CC^\wedge_p$ in \cite[Theorem 2.30]{mathew_neumann_noel_descent}. More precisely, for any presentable stable symmetric monoidal $\infty$-category $\CC$, the $\ell$-completion can be described as the homotopy limit
    \[
    C^\wedge_\ell \simeq \lim_\Delta \left( \Mod_\CC(\Unit/\ell) \rightrightarrows \Mod_\CC(\Unit/\ell \otimes \Unit/\ell) \triplerightarrow{} \cdots \right).
    \]
    The claim follows from the fact that symmetric monoidal functors preserve $\Unit/\ell$.
\end{proof}

\section{Framed motivic spectra}\label{section-framed}
The goal of this section is to extend the Reconstruction Theorem of \cite[Theorem 18]{Hoy_locthm21} to algebraic stacks (see Corollary \ref{corollary-framed-SH-limit}). 
The theorem was proved in \cite[Theorem 3.5.12]{mot_loop_spaces_21} for perfect fields and in \cite[Theorem 18]{Hoy_locthm21} for general schemes.
We start by recalling some facts about the $\infty$-category of framed correspondences 
$\Corr^\fr(\Sch)$ from \cite[Section 4]{mot_loop_spaces_21}. The category $\Corr^\fr(\Sch)$ 
is an $\infty$-category whose objects are $S$-schemes and whose mapping spaces are the 
$\infty$-groupoids $\Corrs^\fr(X,Y)$ defined in \textit{loc. cit.} Section 2.3.

\begin{defn}
    Let $X,Y \in \Sch$. A \textit{(tangentially) framed correspondence} from $X$ to $Y$ over $S$ consists of the following data
    \begin{enumerate}
        \item A span 
        \begin{equation*}
        \begin{tikzcd}
            & Z \ar[dr, "g"] \ar[dl, "f" swap]& \\
             X & & Y
        \end{tikzcd}
        \end{equation*}
        over $S$, where $f$ is finite syntomic.
        \item  A trivialisation $\tau \colon 0 \simeq \clg{L}_f$ of the cotangent complex of $f$ in the $K$-theory $\infty$-groupoid $K(Z)$.
    \end{enumerate}
    We denote by $\Corrs^\fr(X,Y)$ the $\infty$-groupoid of tangentially framed correspondences: 
    \[
    \Corrs^\fr(X,Y) := \colim_{X \leftarrow Z \rightarrow Y} \Map_{K(Z)}(0,\clg{L}_f),
    \]
    where the colimit is taken over all spans as above.
\end{defn}

Let $\XC$ be an algebraic stack.
We define the category of \textit{framed correspondences on $\XC$} as follows. 
We apply the construction from Section 4 
in \textit{loc. cit.} to the pair $(\Sm_\XC,\fsyn)$ and 
with the labeling functor sending a finite syntomic morphism $f$ to the groupoid of trivialisations of its cotangent sheaf. 
Denote the resulting $\infty$-category by $\Corr^\fr(\Sm_\XC)$.
The situation over a stack is slightly different from the situation over a scheme. In particular, there are the following differences from the case of schemes.
\begin{enumerate}
    \item The category $\Sm_\XC$ does not have a final object, hence we cannot speak about the pointed category $\Sm_{\XC,*}$. Nonetheless, non-empty fiber products exist and they are computed as in  \textit{loc. cit.} Section 4.3.
    \item The fiber product on $\Sm_\XC$ does not induce a symmetric monoidal structure on $\Corr^\fr(\Sm_\XC)$ as in the scheme case.
\end{enumerate}

There is a natural functor $\gamma \colon \Sm_\XC \to \Corr^\fr(\Sm_\XC)$, sending a morphism $f \colon X \to Y$ to the correspondence
\begin{equation*}
    \begin{tikzcd}
        & X \ar[dl,"\id" swap] \ar[dr,"f"] & \\
        X && Y.
    \end{tikzcd}
\end{equation*}
This functor is a bijection on objects, commutes with fiber products (if they exist), and induces a pair of adjoint functors
\begin{equation*}
    \hat{\gamma}_* \colon \PSh_\Sigma(\Corr^\fr(\Sm_\XC)) \leftrightarrows \PSh_\Sigma(\Sm_\XC) :\hat{\gamma}^*,
\end{equation*}
where $\PSh_\Sigma(-)$ denotes the full subcategory of presheaves taking finite coproducts to finite products. The functor $\hat{\gamma}_*$ is conservative, because isomorphisms of sheaves are detected on sections. Furthermore, note that $\hat{\gamma}_*$ naturally factors through the pointed presheaf category $\PSh(\Sm_\XC)_*$. 
Indeed, let $F \in  \PSh_\Sigma(\Corr^\fr(\Sm_\XC))$, then for any $X \in \Sm_\XC$, 
the empty correspondence from $X$ to $\varnothing$ induces a map $F(\varnothing) = * \to F(X)$. We denote the resulting functor by
\[
\gamma_* \colon  \PSh_\Sigma(\Corr^\fr(\Sm_\XC)) \to \PSh_\Sigma(\Sm_\XC)_*.
\]
Since $\hat{\gamma}^*$ preserves pointed sheaves, it restricts to a left adjoint to $\gamma_*$, which we denote by $\gamma^*$

\begin{defn}
    We say a presheaf on $\Corr^\fr(\Sm_\XC)$ is \textit{Nisnevich local}, respectively, \textit{$\A^1$-invariant}, if its restriction to $\Sm_\XC$ along $\gamma$ is. Denote the corresponding full subcategories by $\Sh_\Nis(\Corr^\fr(\Sm_\XC))$, $\PSh^{\A^1}(\Corr^\fr(\Sm_\XC))$. Moreover,
    the f\textit{ramed motivic homotopy category} $\Hm^\fr(\XC)$ of $\XC$ is defined as the full-subcategory of $\PSh_{\Sigma}(\Corr^\fr(\Sm_\XC))$ consisting of $\A^1$-invariant and Nisnevich local presheaves.
\end{defn}

\begin{theorem} \label{thm-framed-h}
    Let $\XC$ be an algebraic stack with smooth Nisnevich cover $p_0 \colon X \to \XC$. Then the natural functor
    \begin{equation*}
        p^* \colon \Hm^\fr(\XC) \to \lim_\Delta \Hm^\fr(X_\bullet)
    \end{equation*}
    is an equivalence of $\infty$-categories, where $p \colon X_\bullet \to \XC$ denotes the \v{C}ech nerve of $X \to \XC$.
\end{theorem}

\begin{proof}
As in Lemma \ref{lem_comp_functor}, we have a pair of adjoint functors
\begin{equation*}
    p^{\fr,^*} \colon \PSh_\Sigma(\Corr^\fr(\Sm_\XC)) \leftrightarrows \lim_\Delta 
    \PSh_\Sigma(\Corr^\fr(\Sm_{X_n}) : p^\fr_*.
\end{equation*}
Assume for now that $\hat{\gamma}_*$ commutes with $p^{\fr,*}$ and $p^\fr_*$, i.e., 
that the following diagram commutes
\begin{equation*}
\begin{tikzcd}
    \PSh_\Sigma(\Corr^\fr(\Sm_\XC)) \ar[d, shift left, "p^{\fr,*}"] \ar[r, "\hat{\gamma}_*"]
    &\PSh(\Sm_\XC) \ar[d, shift left,"p^*"] \\
    \lim_\Delta \PSh_\Sigma(\Corr^\fr(\Sm_{X_\bullet})) \ar[u, shift left, "p^\fr_{*}"] 
    \ar[r, "(\hat{\gamma}_{\bullet})_*"] & \lim_\Delta \PSh(\Sm_{X_\bullet}) \ar[u, shift left, "p_*"].
\end{tikzcd}
\end{equation*}
Under this assumption, it follows that $p^\fr_*$ preserves Nisnevich sheaves and $\A^1$-invariant presheaves, thus $p^{\fr,*} \dashv p^\fr_*$ restricts to an adjunction
\[
p^{\fr,*} \colon \Hm^\fr(\XC) \leftrightarrows \lim_\Delta \Hm^\fr(X_\bullet): p^\fr_*
\]
It remains to show that the unit and counit of the above adjunction are equivalences.
We give the argument for the unit map.
Let $F \in \Hm^\fr(\XC)$, then applying $\hat{\gamma}_*$ to 
\[
F \to p^\fr_* p^{\fr,*} F
\]
gives the unit map 
\[
\hat{\gamma}_*F \to p_* p^* \hat{\gamma}_*F
\]
of the adjunction $p^* \colon \Hm(\XC) \dashv \lim_\Delta \Hm(X_\bullet) : p_*$,
which is an equivalence by Theorem \ref{thm-sh}. 

Finally, let us prove that $\hat{\gamma}_*$ commutes with $p^{\fr,*}$ and $p^{\fr}_*$. Since $\hat{\gamma}_*$ commutes with limits, it suffices to prove that the diagram
\begin{equation*}
\begin{tikzcd}
    \PSh_\Sigma(\Corr^\fr(\Sm_\XC)) \ar[d, shift left, "p^{\fr,*}_n"] \ar[r, "\hat{\gamma}_*"]
    &\PSh(\Sm_\XC) \ar[d, shift left,"p^*"] \\
    \PSh_\Sigma(\Corr^\fr(\Sm_{X_n})) \ar[u, shift left, "p^\fr_{n,*}"] 
    \ar[r, "\hat{\gamma_n}_*"] & \PSh(\Sm_{X_n}) \ar[u, shift left, "p_*"]
\end{tikzcd}
\end{equation*}
commutes for every $n$. This follows by direct computation. Let $F \in \PSh(\Corr^\fr(\Sm_\XC))$ and let $U \in \Sm_{X_n}$, then
\begin{align*}
    \hat{\gamma}_{n,*} p^{\fr,*}_n F (U) = p^{\fr,*}_n F(U) = F(p_{n,\#} U)
    = \hat{\gamma}_* F(p_{n,\#}U) = p_n^* \hat{\gamma}_* F(U).
\end{align*}
and similarly, for $F \in \PSh(\Corr^\fr(\Sm_{X_n}))$ and $U \in \Sm_\XC$ we have
\[
\hat{\gamma}_* p^\fr_{n,*} F (U) = p^\fr_{n,*} F(U) = F(U\times_\XC X_n) =
\hat{\gamma}_{n,*} F(U \times_\XC X_n) = p_* \hat{\gamma}_{n,*} F (U).
\]
\end{proof}

The categories $\Hm^\fr(X_n)$ carry a symmetric monoidal structure \cite[Section 3.2]{mot_loop_spaces_21}, which is defined using the monoidal structure on $\Corr^\fr(\Sm_{X_n})$. We can use the equivalence
\[
\Hm^\fr(\XC) \cong \lim_\Delta \Hm^\fr(X_\bullet)
\]
to endow $\Hm^\fr(\XC)$ with a symmetric monoidal structure $\otimes$ such that the above equivalence becomes symmetric monoidal. With this monoidal structure, the functor 
\[
\gamma^* \colon \Hm(\XC)_* \to \Hm^\fr(\XC)
\]
also becomes symmetric monoidal. For example, for $X,Y \in \Sm_\XC$, we have an equivalence
\[
\gamma^*(X_+) \otimes \gamma^*(Y_+) \simeq \gamma^*((X\times_\XC Y)_+)
\]
in $\Hm^\fr(\XC)$.

\begin{defn}
Let $\XC$ be as above. Let $T:= \A^1/\A^1 - 0 \in \Hm(\XC)_*$ denote the Tate object. The \textit{stable framed motivic homotopy category} of $\XC$ is defined as the $T^\fr := \gamma^*(T)$-stabilisation of $\Hm^\fr(\XC)$,
\[
\SH^\fr(\XC) := \Hm^\fr(\XC)\left[(T^\fr)^{-1}\right].
\]
\end{defn}

\begin{cor}\label{corollary-framed-SH-limit}
    Let $X_\bullet \to \XC$ be as in Theorem \ref{thm-framed-h}. Then the functor $p^{\fr,*}$ induces an equivalence
    \begin{equation*}
        p^{\fr,*} \colon \SH^\fr(\XC) \to \lim_\Delta \SH^\fr(X_{\bullet})
    \end{equation*}
    of symmetric monoidal $\infty$-categories. Moreover, the functor $\gamma^*$ induces an equivalence
    \begin{equation*}
        \gamma^{*} \colon \SH(\XC) \to \SH^\fr(\XC)
    \end{equation*}
    of symmetric monoidal $\infty$-categories.
\end{cor}
\begin{proof}
    The same argument as in the proof of Theorem \ref{thm-sh} applied to Theorem \ref{thm-framed-h} shows that
    $p^{\fr,*}$ induces an equivalence of symmetric monoidal categories
    \begin{equation*}
        p^{\fr,*} \colon \SH^\fr(\XC) \to \lim_\Delta \SH^\fr(X_\bullet).
    \end{equation*}
    Moreover, there is a commutative diagram of the form
    \begin{equation*}
    \begin{tikzcd}
        \SH(\XC) \ar[r,"\gamma^*"] \ar[d] & \SH^\fr(\XC) \ar[d] \\
        \lim_\Delta \SH(X_\bullet) \ar[r, "\gamma^*_\bullet"] & \lim_\Delta \SH^\fr(X_\bullet),
    \end{tikzcd}
    \end{equation*}
    where the vertical arrows are the equivalences and the bottom arrow is an equivalence by \cite[Theorem 18]{Hoy_locthm21}. 
\end{proof}

The following result extends \cite[Lemma 16]{Hoy_locthm21} to algebraic stacks. It states that the pullback functor along a morphism of algebraic stacks respects forgetting framed transfers.

\begin{cor}\label{corollary-pullback-commutes-with-forgetting-transfers}
    Let $f: \XC \rightarrow \YC$  be a morphism of algebraic stacks. Then the following diagram commutes,
    \begin{center}
    \begin{tikzcd}
        \Hm^\fr(\YC)\arrow[r,"\gamma_*"]\arrow[d,"f^*"] & \Hm(\YC)\arrow[d,"f^*"]\\
        \Hm^\fr(\XC)\arrow[r,"\gamma_*"] & \Hm(\XC)
    \end{tikzcd}
    \end{center}
    Similarly for $\SH$.
\end{cor}
\begin{proof}
    This follows from Theorem \ref{thm-framed-h}, Corollary \ref{corollary-framed-SH-limit} and \cite[Lemma 16]{Hoy_locthm21}.
\end{proof}

One consequence of having a model for $\SH^\fr(\XC)$ for an algebraic stack $\XC$ is that the motivic cohomology spectrum can be described as the framed suspension of the constant sheaf $\Z$, exactly as in the case of schemes.

\begin{cor}[The motivic cohomology spectrum over $\XC$]\label{corollary-framed-motivic-cohomology-spectrum}
    Let $\XC$ be an algebraic stack. Consider the constant sheaf $\Z_\XC$ on $\Sm_\XC$. Then $\gamma_*\Sigma^\infty_{T,\fr}\Z_\XC\simeq H\Z^\mot_\XC$.
\end{cor}
\begin{proof}
    The corollary follows from the fact that $\gamma_*$ and the framed suspension functor respect pullbacks. Let $\Sigma^\infty_{T,\fr}\Z_\XC \in \SH^\fr (\XC)$ be the framed suspension of the constant sheaf $\Z$ on $\Sm_\XC$.

    Notice that $H\Z^\mot_\XC := f^*H\Z^\mot$ along the structure map $f\colon  \XC\rightarrow \Spec (\Z)$, where $H\Z^\mot$ is Spitzweck's motivic cohomology spectrum (see \cite{Spitzweck_P1_spectrum}). Indeed, by \cite[Theorem 21]{Hoy_locthm21}, the functor $\gamma_*\colon  \SH^\fr (\Z)\rightarrow \SH(\Z)$
    induces an isomorphism $\gamma_*\Sigma^\infty_{T,\fr}\Z \simeq H\Z^\mot$. Thus, $f^*\gamma_* \Sigma^\infty_{T,\fr}\Z \simeq H\Z^\mot_\XC$.
    On the other hand, $\gamma_* f^*\Sigma^\infty_{T,\fr}\Z\simeq \gamma_* \Sigma^\infty_{T,\fr}f^*\Z\simeq \gamma_* \Sigma^\infty_{T,\fr}\Z_\XC$. Corollary \ref{corollary-pullback-commutes-with-forgetting-transfers} now gives us the required equivalence.
\end{proof}

\section{Application to Deligne-Mumford stacks}\label{section-coarse-space-morphism}
Let $G$ be a finite group of order $N$. Let $S$ be a scheme such that $N \in \clg{O}_S^\times$ and let $X$ be a quasi-projective $S$-scheme. 
Assume that $G$ acts on $X$ such that the quotient stack $[X/G]$ is smooth.
By \cite[Theorem 6.12]{rydh_geo_quot}, there exists a scheme $Y$, the coarse space of $[X/G]$, and a morphism $\pi \colon [X/G] \to Y$ which is initial among morphisms to algebraic spaces and a universal homeomorphism on the underlying topological space.

\begin{theorem}\label{thm-quot-stack}
    Let $X$ and $G$ be as above. Then the induced map
    \begin{equation*}
        \pi^* \colon \SH_\et(Y, \Z[1/N]) \to \SH_\et([X/G],\Z[1/N])
    \end{equation*}
    is fully faithful.
\end{theorem}

\begin{remark}
    The above functor is almost never essentially surjective. For example, if we consider
    the example $X = S = \Spec k$, where $k$ is a field of characteristic zero, and $G$ acts trivially on $X$.
    Then, by Theorem \ref{thm-sh}, a motivic sheaf on $[\Spec k/G] = BG$ is a motivic sheaf on $\Spec k$ with a group action of $G$. In this case, the essential image of the functor
    \[\pi^*\colon \SH_{\et}(\Spec k, \Z[1/N])\rightarrow \SH_{\et}(BG, \Z[1/N])\]
    associated to the coarse map $\pi: BG\rightarrow \Spec k$, is given by motivic sheaves of $\Z[1/N]$-modules with \textit{trivial} $G$-action.
\end{remark}

\begin{proof}[Proof of Theorem \ref{thm-quot-stack}]
Let $p \colon X \times G^{\bullet} \to [X/G]$ be the nerve of the group action $X \times G \to X$. This is an \'etale hypercover $[X/G]$.
By Corollary \ref{et/coefficients-version} the natural map 
\[
p^* \colon \SH_\et([X/G], \Z[1/N]) \to \lim_{\Delta} \SH_\et(X \times G^\bullet, \Z[1/N])
\]
is an equivalence.
The composition $p^* \circ \pi^*$ is given by
\begin{align*}
    p^* \circ \pi^* \colon \SH_\et(Y, \Z[1/N]) &\to \lim_\Delta \SH_\et(X \times G^\bullet, \Z[1/N])\\
    F &\mapsto [[n] \mapsto q_n^* F],
\end{align*}
where $q_n \colon X\times G^n \to Y$ denotes the composition $q_n \colon X \times G^n \to [X/G] \to Y$.

Let $E,F \in \SH(Y,\Z[1/N])$. We have to show that the natural map
\begin{equation*}
    \Hom_{\SH_\et(Y,\Z[1/N])}(E,F) \to \Hom_{\SH_\et([X/G],\Z[1/N])}(\pi^*E, \pi^*F)
\end{equation*}
is an equivalence. 
We view the mapping spaces as complexes in $D(\Z[1/N])$ via the Dold-Kan correspondence.
Since $\SH_\et(Y,\Z[1/N])$ is generated by compact, strongly dualizable objects, it suffices to treat the case $E = \Unit_Y$. 
We calculate
\begin{align*}
    \Hom_{\SH_\et([X/G],\Z[1/N])}(\Unit_{[X/G]}, \pi^* F) &\simeq
    \Hom_{\SH_\et([X/G],\Z[1/N])}(\Unit_{[X/G]}, p_* p^* \pi^* F)\\ &\simeq^{\ref{lem_comp_functor}} 
    \Hom_{\SH_\et([X/G],\Z[1/N])}(\Unit_{[X/G]}, \lim_{\Delta} p_{n,*} q_n^* F) \\
    &\simeq  \lim_\Delta \Hom_{\SH_\et(X\times G^n ,\Z[1/N])}(\Unit_{X\times G^n}, q_n^* F) \\
    &\simeq \lim_\Delta q_0^*F(X\times G^\bullet).
\end{align*}
As $|G| = N$ is invertible in $\Z[1/N]$, the above limit collapses to $(q_0^*F(X))^G$ (this can be seen, for example, by analysing the associated descent spectral sequence).
But this is precisely equal to $F(Y)= \Hom(\Unit_Y, F)$.
\end{proof}

\begin{cor}
    Let $X/S$ and $G$ be as above. Assume that for every closed point $y \in S$, the residue field $k(y)$ has no real embeddings, then the natural map
    \begin{equation*}
        \pi^* \colon \SH(Y, \Q) \to \SH([X/G],\Q)
    \end{equation*}
    is fully faithful.
\end{cor}
\begin{proof}
Under these assumptions, we have the following equivalences,
\[
\SH_\et(S, \Q) \simeq \DM_\et(S, \Q) \simeq \DM(S, \Q) \simeq \SH(S, \Q),
\]      
\end{proof}

\begin{theorem}\label{theorem-coarse-space-map-fully faithful}
	Let $\XC$ be a separated Deligne-Mumford stack with coarse space morphism $\pi: \XC\rightarrow Y$. Then the pullback functor
	\[\pi^* \colon \SH_{\et}(Y, \Q) \to \SH_{\et}(\XC,\Q)\]
	is fully faithful
\end{theorem}

\begin{proof}
	By \cite[Lemma 2.2.3]{abramovich_compactification}, \'{e}tale locally on $Y$, we are reduced situation where $\XC=[X/G]$ for a finite \'{e}tale group scheme $G$ over $Y$. As $\SH_{\et}$ satisifies \'{e}tale descent, we are done by Theorem \ref{thm-quot-stack}.
\end{proof}

\begin{remark}
    In the above corollary, if $\XC$ is a quasi-compact separated Deligne-Mumford stack, then there exists an $N\in \Z$ sufficiently large such that the orders of the stabilisers of $\XC$ are bounded by $N$. Then a similar argument shows that, in fact, $\pi^*$ is fully faithful for $\SH_\et(-,\Z[1/N])$.
\end{remark}

\section{Cocomplete coefficient systems on algebraic stacks}\label{section-coefficient-systems}

In this section, we discuss the theory of cocomplete coefficient systems \cite{DG_six_functor} due to Drew-Gallauer in the context of algebraic stacks. This is the $\infty$-categorical reformulation of Ayoub's stable homotopy 2-functors \cite{ayoubthesisI, ayoubthesisII}.

Informally, the main theorem in \cite{DG_six_functor} says that $\SH$ is the universal six functor formalism on the category of schemes which satisfies Nisnevich descent and $\A^1$-invariance, and $\PR^1$-stability.
Any six-functor formalism that is extended to stacks from schemes (for example, by methods of \cite{chowdhury2024}) via Nisnevich descent should satisfy a similar property for representable morphisms for formal reasons. In this section, we extend the theory to include non-representable morphisms.

The following definition is an extension of \cite[Definition 7.5]{DG_six_functor} to algebraic stacks. 

\begin{defn}
    Let $\Stk_S$ be the category of algebraic stacks over $S$. A functor $C\colon  \Stk_S\rightarrow \CAlg(\Cat^{\stb}_\infty)$ taking values in symmetric monoidal stable $\infty$-categories with exact symmetric monoidal functors is called a \textit{coefficient system} if it satisfies the following properties:
    \begin{enumerate}
        \item 
        \begin{itemize}
            \item[(a)] For every $f\colon \XC\rightarrow \YC$ in $\Stk_S$, the pullback functor $f^*\colon C(\YC)\rightarrow C(\XC)$ admits a right adjoint $f_*\colon  C(\XC)\rightarrow C(\YC)$.
            \item[(b)] For every $\XC\in \Stk_S$, the symmetric monoidal structure on $C(\XC)$ is closed.
        \end{itemize}
        \item For every smooth morphism $\XC\rightarrow\YC$ in $\Stk_S$, the functor $f^*\colon C(\YC)\rightarrow C(\XC)$ admits a left adjoint $f_\#$ such that:
        \begin{itemize}
            \item[(a)] For each cartesian square in $\Stk_S$
            \begin{center}
                \begin{tikzcd}
                    \XC'\arrow[d,"g"]\arrow[r,"q"] & \XC\arrow[d,"f"]\\
                    \YC'\arrow[r,"p"]& \YC
                \end{tikzcd}
            \end{center}
                the exchange transformation $g_\#q^*\rightarrow p^*f_\#$ is an equivalence.
            \item[(b)] The exchange transformation $f_\#(-\otimes f^*(-))\rightarrow f_\#(-)\otimes -$ is an equivalence.
        \end{itemize}
        \item \begin{itemize}
            \item[(a)] For every closed immersion $i\colon \mathcal{Z}\rightarrow\XC$ in $\Stk_S$ with complementary open immersion $j\colon \mathcal{U}\rightarrow\XC$, the following square is cartesian in $\Cat_\infty^\stb$
        \begin{center}
            \begin{tikzcd}
                C(\mathcal{Z})\arrow[d]\arrow[r,"i_*"] & C(\XC)\arrow[d,"j^*"]\\
                0\arrow[r] & C(\mathcal{U}).
            \end{tikzcd}
        \end{center}
        \item[(b)] There exists a smooth Nisnevich covering $p_0 \colon X\rightarrow\XC$ for some $X\in \Sm_S$ whose \v{C}ech nerve induces an equivalence
        \[p^*\colon C(\XC)\overset{\sim}{\rightarrow} \lim_\Delta C(X_\bullet).\]
        \end{itemize}
        \item For $\XC\in \Stk_S$, let $\pi_{\A^1}\colon \A^1_\XC \rightarrow \XC$ denote the projection map with the zero section $s\colon X\rightarrow \A^1_\XC$, then
        \begin{itemize}
            \item[(a)] $\pi^*_{\A^1}\colon  C(\XC)\rightarrow C(\A^1_\XC)$ is fully faithful.
            \item[(b)] The composite $\pi_{\A^1,\#}\circ s_* \colon C(\XC)\rightarrow C(\XC)$ is an equivalence.
        \end{itemize}
    \end{enumerate}
\end{defn}

\begin{remark}
    The above definition differs from \cite{DG_six_functor} in that we need to force Nisnevich descent for schemes to stacks by adding the condition $3. (b)$ (see also \cite{chowdhury2024} for a similar statement). The reason for this is that there exist symmetric monoidal functors $C\colon  \Stk_S\rightarrow \CAlg(\Cat_\infty^\stb)$ which satisfy $3. (a)$ but are not Kan extended from schemes.
    For example, the genuine equivariant theory of Hoyois satisfies all the axioms above for representable morphisms. However, it does not satisfy $3. (b)$.
\end{remark}

\begin{remark}
    Another example of a symmetric monoidal functor which does not satisfy Nisnevich descent is given by the functor $\SH_{\mathcal{I}}\colon  \XC\mapsto \SH(\mathcal{I}_\XC)$, where $\clg{I}_\XC$ is the inertia stack of $\XC$. It satisfies $3. (a)$ but not $3.(b)$.
    Moreover, $\SH_{\mathcal{I}}$ satisfies $2.$ only for \textit{stabiliser preserving} smooth morphisms.
    We think of this functor as ``sitting in between" the Kan extended $\SH$ and the genuine equivariant $\SH$. We will explore these ideas in future work.
\end{remark}

\begin{defn}
Let $\Stk_S$ be the category of algebraic stacks over $S$.
A \textit{cocomplete coefficient system }on $\Stk_S$ is a functor $C\colon  \Stk_S\rightarrow \CAlg(\Cat_\infty^{c,\stb})$ taking values in cocomplete symmetric monoidal stable $\infty$-categories with cocontinuous symmetric monoidal functors such that the composition with the forgetful functor $\Cat^{c,\stb}_\infty\rightarrow \Cat_\infty^\stb$ is a coefficient system.
\end{defn}

\begin{defn}
Let $\STCOSY^c_S$ be the category of cocomplete coefficient systems on $\Stk_S$. A morphism $\mathfrak{R}\colon  C\rightarrow D$ between cocomplete coefficient systems is an exact symmetric monoidal natural transformation such that for any morphism $f\colon  \XC\rightarrow \YC \in \Stk_S$, the exchange transformations 
$\theta_f\colon  f_{*}\circ \mathfrak{R}_{\XC} \rightarrow \mathfrak{R}_{\YC}\circ f_{*}$ and $\theta_f^{-1}\colon  \mathfrak{R}_{\YC}\circ f_{\#}\rightarrow f_{\#}\circ \mathfrak{R}_{\XC}$ are invertible.
\end{defn}

\begin{remark}
    As in \cite{DG_six_functor}, we denote by $\COSY^c_S$ the category of cocomplete coefficient systems on $\Sch_S$.
    The inclusion of categories $\Sch_S\hookrightarrow \Stk_S$ induces a functor $i\colon  \STCOSY^c_S\rightarrow \COSY^c_S$.
\end{remark}

\begin{lemma}\label{lemma-extending-elements-cosy-to-stcosy}
    Let $C\in \COSY^c_S$. Then $C$ extends to a cocomplete coefficient system on $\Stk_S$.
\end{lemma}
\begin{proof}
    We use the formalism described in \cite[Appendix A.2]{chowdhury2024} to prove this.
    Let $\Corr (\Sch_S)_{\mathrm{sm,all}}$ be the category of correspondences whose objects are $S$-schemes, and morphisms are given by spans $X\overset{f}{\leftarrow} Z\overset{g}{\rightarrow} Y$ such that $f,g$ are morphisms of $S$-schemes and $g$ is smooth.

    If $C$ is cocomplete coefficient system on $\Sch_S$, then we have an induced a functor
    \[C^*_\# \colon \Corr (\Sch_S)_{\mathrm{sm,all}}\rightarrow \CAlg(\Cat_\infty)\] that sends $X$ to $C(X)$, and a span $X\overset{f}{\leftarrow} Z\overset{g}{\rightarrow} Y$ to $g_\# f^*$. This describes a six functor formalism for the pair $(f^*, f_\#)$ in the sense of \cite[Appendix A.5]{mann2022sixfun}

    Let $\mathrm{smrep}$ and $\mathrm{stsm}$ denote the classes of smooth representable morphisms, respectively, of all smooth morphisms on $\Stk_S$. Then, by \cite[Theorem A.16]{chowdhury2024}, $C$ extends to a functor \[C'^*_\#\colon  \Corr (\Stk_S)_{\mathrm{smrep,all}}\rightarrow \CAlg (\Cat_\infty^\stb).\] Furthermore, by \textit{loc. cit.} Theorem A.19, $C'^*_\#$ extends to a functor \[\overline{C}^*_\#\colon  \Corr (\Stk_S)_{\mathrm{stsm,all}}\rightarrow \CAlg (\Cat_\infty^\stb).\]

    We have an inclusion $\Stk_S^\op\hookrightarrow \Corr (\Stk_S)_{\mathrm{stsm,all}}$ that sends any morphism $f\colon  \XC\rightarrow \YC$ to the correspondence $\YC \overset{f}{\leftarrow} \XC \overset{id}{\rightarrow} \XC$. The composition gives us a functor $\overline{C}\colon  \Stk_S^\op\rightarrow \CAlg (\Cat_\infty^\stb)$ such that $\XC\mapsto \overline{C}^*_\#(\XC)$ and a morphism $f\mapsto f^*$. We claim that $\overline{C}$ is a cocomplete coefficient system on $\Stk_S$.

    Properties (1) and (2) are satisfied due to Theorems A.16, A.19 of \textit{loc. cit.}, and the discussion in Appendix A.2 regarding smooth base change and the projection formula.
    
    Properties (3) and (4) follow from Nisnevich descent. 
    For this, choose a simplicial resolution $p\colon  X_{\bullet}\rightarrow \XC$ (for example, the nerve of a smooth Nisnevich covering $X_0\rightarrow\XC$).
    
    To prove property (3), let $i\colon \mathcal{Z}\rightarrow \XC$ be a closed substack with complementary open $j\colon \mathcal{U}\rightarrow \XC$. Restricting $X_{\bullet}$ to these substacks gives us simplicial resolutions $p_{\mathcal{Z}}\colon  Z_{\bullet}\rightarrow \mathcal{Z}$  and $p_{\mathcal{U}}\colon  U_{\bullet}\rightarrow\mathcal{U}$ such that the square
    \begin{center}
            \begin{tikzcd}
                C(Z_n)\arrow[d]\arrow[r,"i_{n*}"] & C(X_n)\arrow[d,"j_n^*"]\\
                0\arrow[r] & C(U_n)
            \end{tikzcd}
        \end{center}
    is Cartesian in each simplicial degree $n$. As $i$ is a closed immersion, $i_{n*}$ commutes with pullback. Then, taking limits gives us the required Cartesian square.

    Similarly for property (4), the map $\pi^*_{\A^1}\colon  C(X_n)\rightarrow C(\A^1_{X_n})$ is fully faithful in each degree $n$. Thus, by Nisnevich descent, so is the map $\pi^*_{\A^1}\colon  C(\XC)\rightarrow C(\A^1_\XC)$. Finally, the map $\pi_{\A^1 \#}\circ s_* \colon C(X_n)\rightarrow C(X_n)$ is an equivalence for each $n$, and $\pi_{\A^1 \#}$ and $s_*$ behave well with respect to base change (by smooth base change formula and the fact that $s$ is a closed immersion). Thus, $\pi_{\A^1,\#}\circ s_* \colon C(\XC)\rightarrow C(\XC)$ by Nisnevich descent.
\end{proof}

\begin{lemma}\label{lemma-extending-realisation-to-stacks}
    Let $\mathfrak{R}\colon  C\rightarrow D \in \COSY^c_S$ be a morphism of cocomplete coefficient systems. Then it extends to a morphism in $\STCOSY^c_S$.
\end{lemma}
\begin{proof}
    Note that $C$ and $D$ satisfy decent along $f^*$ and codescent along $f_\#$.
    We note that \cite[Proposition A.19 and A.23]{chowdhury2024}, are functorial in $\mathfrak{R}$.
    Thus, by descent and codescent, $\mathfrak{R}$ extends to a morphism in $\STCOSY_S^c$, and that $\theta_f$ and $\theta_f^{-1}$ are invertible.
\end{proof}

\begin{prop}\label{prop-equivalence-cosy-stcosy}
    The functor $i\colon \STCOSY^c_S\rightarrow \COSY^c_S$ is an equivalence of $\infty$-categories.
\end{prop}
\begin{proof}
    By the previous two lemmas, we have a functor $j\colon  \COSY^c_S\rightarrow \STCOSY^c_S$, which is a right adjoint of $i$. 
    Thus, the unit and counit of the adjunction give natural transformations $\id \to j \circ i$ and $i\circ j \to \id$.
    We claim that these are equivalences.
    For any $C\in \STCOSY^c_S$ we have that $j\circ i (C(\XC))\simeq \lim_{\Delta} C(X_{\bullet})$, for some Nisnevich local equivalence $X_{\bullet}\rightarrow \XC$.
    As any $C\in \STCOSY^c_S$ satisfies Nisnevich descent along $f^*$, we have an isomorphism $j\circ i (C(\XC))\to C(\XC)$. Further, by uniqueness of adjoints, the adjoint pairs $f_\# \dashv f^*$ (for $f$ smooth) and $f^*\dashv f_*$ are uniquely determined up to contractible choices. Thus, $j\circ i$ is an equivalence.

    Similarly, if $D\in \COSY^c_S$, then the extension $j(D)$ agrees with $D$ on $\Sch_S$, hence, the adjunctions $f_\# \dashv f^*$ (for $f$ smooth) and $f^*\dashv f_*$ are also preserved. Consequently, $i \circ j(D) \to D$ is an equivalence.
\end{proof}

\begin{cor}\label{corollary-SH-initial}
   The object $\mathbf{SH}$ is initial in $\STCOSY^c_S$.
\end{cor}
\begin{proof}
    The functor $\mathbf{SH}$ on $\Sch_S$ is initial in $\COSY^c_S$ by \cite[Theorem 7.14]{DG_six_functor}. Thus, by the previous proposition, $\mathbf{SH}$ is initial in $\STCOSY^c_S$. 
\end{proof}

\begin{remark}
    A statement similar to the above corollary is also proved in \cite{chowdhury2024}.
\end{remark}

By \cite{ayoubthesisI, ayoubthesisII}, any cocomplete coefficient system on $\Sch_S$ gives rise to a six-functor formalism. Together with the results in \cite{chowdhury2024}, we have a similar statement for $\Stk_S$.

\begin{prop}
    Let $C\in\STCOSY_S^c$ be a cocomplete coefficient system. Then for a locally finite type map $f\colon  \XC\rightarrow\YC$ of algebraic stacks, $C$ admits adjoints $f_!\dashv f^!$, satisfying the conclusion of \cite[Theorem 4.26]{chowdhury2024}.
\end{prop}
\begin{proof}
    Note that the statement is true for the restriction $C_{|\Sch_S}$.
    Moreover, $C_{|\Sch_S}^*$ is a Nisnevich sheaf, while $C_{|\Sch_S !}$ is a Nisnevich cosheaf.
    Thus, the argument for $\SH$ presented in \cite{chowdhury2024} works in this case and \cite[Theorem 4.26]{chowdhury2024} holds for $C$.
\end{proof}

\begin{prop}\label{prop-extending-realisation-commutes-with-shriek-to-stacks}
    Let $\mathfrak{R}\colon  C\rightarrow D$ be a morphism in $\STCOSY^c_S$. Then $\mathfrak{R}$ commutes with $f_!$, i.e, the natural transformation $\rho_f\colon  f_!\circ \mathfrak{R}\rightarrow \mathfrak{R}\circ f_!$ is invertible.
\end{prop}
\begin{proof}
    By Theorem 3.4 of \cite{ayoub_betti}, $\rho_f$ is invertible for any morphism $f\colon  X\rightarrow Y$ in $\Sch_S$. 
    
    By the argument in \cite[Proposition A.23]{chowdhury2024} applied to $\mathfrak{R}$, we see that for any morphism $f\colon  \XC\rightarrow\YC$, the natural transformation $\rho_f$ is constructed as a colimit over $\rho_{f_i}$ which are invertible in each simplicial degree. Thus, by the universal property of colimits, $\rho_f$ is also invertible.
\end{proof}

\subsection{\'Etale motives and realisation}

In this section, we extend Ayoub's \'{e}tale realisation functor \cite{ayoub_etreal} to algebraic stacks. In what follows, we freely use the notation and hypothesis in \textit{loc. cit.}

Let $\Lambda$ be a ring and $I\subset \Lambda$ be an ideal. 
Let $D_\et(\XC,\Lambda/I)$ and $D_\et(\XC,\Lambda_I)$ denote the category of \'{e}tale and $I$-adic sheaves on $\Sm_\XC$. As both $D_{\et}(-,\Lambda/I)$ and $D_{\et}(-,\Lambda_I)$ satisfy \'{e}tale hyperdescent, we have equivalences of $\infty$-categories,
\begin{align*}
    D_\et(\XC,\Lambda/I)&\rightarrow  \underset{\Delta}{\lim}\, D_\et(X_{\bullet},\Lambda/I),\\
    D_\et(\XC,\Lambda_I)&\rightarrow \underset{\Delta}{\lim}\, D_\et(X_{\bullet},\Lambda_I),
\end{align*}
where $X_\bullet$ is the \v{C}ech nerve of a smooth cover $p: X\rightarrow\XC$. 

Let $\mathfrak{R}^\et\colon  \DA_\et(\XC, \Lambda)\rightarrow D_\et(\XC,\Lambda/I)$ and $\mathfrak{R}^\et_I\colon  \DA_\et(\XC, \Lambda)\rightarrow D_\et(\XC,\Lambda_I)$ be Ayoub's \'{e}tale realisation and $I$-adic \'{e}tale realisation functors right Kan-extended to stacks.

\begin{theorem}
    The functors $\mathfrak{R}^{\et}\colon  \DA_{\et}(\XC, \Lambda)\rightarrow D_{\et}(\XC,\Lambda/I)$ and $\mathfrak{R}^\et_I\colon  \DA_\et(\XC, \Lambda)\rightarrow D_\et(\XC,\Lambda_I)$ commute with Grothendieck's six operations.
\end{theorem}
\begin{proof}
    The theorem holds for schemes by \cite{ayoub_etreal}. We need to show that it continues to hold after right Kan extending to stacks.

    Firstly, observe that on $\Sch_S$, both $D_\et(-,\Lambda/I)$ and $D_\et(-,\Lambda_I)$ lie in $\COSY^c_S$. Thus, by Proposition \ref{prop-equivalence-cosy-stcosy}, they also lie in $\STCOSY^c_S$. Then Proposition \ref{prop-extending-realisation-commutes-with-shriek-to-stacks} extends the exceptional functors and, therefore, Grothendieck's six operations for
    \'{e}tale cohomology extend to algebraic stacks.
    
    By \textit{loc. cit.}, $\mathfrak{R}^\et|_{\Sch_S}$ and $\mathfrak{R}^\et_I|_{\Sch_S}$ induce a morphism in $\COSY^c_S$. Thus, by Lemma \ref{lemma-extending-realisation-to-stacks}, $\mathfrak{R}^\et$ and $\mathfrak{R}^\et_I$ are morphisms in $\STCOSY^c_S$. Finally, Proposition \ref{prop-extending-realisation-commutes-with-shriek-to-stacks} shows that the realisation functors commute with $f_!$, which finishes the proof.
\end{proof}

\begin{remark}
    Let $\DA^c_{\et}(\XC,\Lambda)$ be the subcategory category of compact objects in $\DA_{\et}(\XC,\Lambda)$. This category admits a dualising sheaf and satisfies Verdier duality. Then under the \'{e}tale realisation functor, this yields a dualising sheaf in the category of \'{e}tale sheaves, and we recover the six functor formalism as described in Olsson-Laszlo \cite{laszlo_olsson_I} and \cite{laszlo_olsson_II}.
\end{remark}